%% file: dof-partly-smooth.tex
\documentclass[smallextended,natbib,runningheads]{svjour3-mod}
\journalname{Annals of the Institute of Statistical Mathematics}
\pdfoutput=1
\usepackage[utf8]{inputenc}
\usepackage{subfigure}
\usepackage{mystyle}

\newcommand{\CondInj}[2]{\Cc_{#1,#2}}

\newif\ifREL
\RELtrue

\usepackage{graphicx}
\graphicspath{{./images/}}

\title{The Degrees of Freedom of Partly Smooth Regularizers}

\author{Samuel~Vaiter \and
        Charles~Deledalle \and
        Jalal~Fadili \and
        Gabriel~Peyr\'{e} \and
        Charles~Dossal
}

\institute{Samuel~Vaiter, Gabriel~Peyr\'{e}  \at
              CEREMADE, CNRS, Universit\'{e} Paris-Dauphine, Place du Mar\'{e}chal De Lattre De Tassigny, 75775 Paris Cedex 16, France \\
              \email{\{samuel.vaiter,gabriel.peyre\}@ceremade.dauphine.fr}           %  \\
           \and
           Jalal~Fadili \at
              GREYC, CNRS-ENSICAEN-Universit\'{e} de Caen, 6, Bd du Mar\'{e}chal Juin, 14050 Caen Cedex, France \\
              \email{Jalal.Fadili@greyc.ensicaen.fr}   
           \and
           Charles~Deledalle, Charles~Dossal \at
              IMB, CNRS, Universit\'{e} Bordeaux 1, 351, Cours de la lib\'{e}ration, 33405 Talence Cedex, France \\
              \email{\{charles.deledalle,charles.dossal\}@math.u-bordeaux1.fr}
           }

\date{ }
\begin{document}

\maketitle
\vspace{-2.5cm}

\input{sections/abstract}

\input{sections/intro}
\input{sections/notations}
\input{sections/prgauges}
\input{sections/local}
\input{sections/well}
\input{sections/gsure}
\input{sections/experiments}
\input{sections/proofs}
\input{sections/conclusion}

\appendix 

\input{sections/o-minimal}

% references section
\bibliographystyle{spbasic}
\bibliography{bibliography} 

\end{document}

%% file: sections/abstract.tex
\begin{abstract}
In this paper, we are concerned with regularized regression problems where the prior regularizer is a proper lower semicontinuous and convex function which is also partly smooth relative to a Riemannian submanifold. This encompasses as special cases several known penalties such as the Lasso ($\lun$-norm), the group Lasso ($\lun-\ldeux$-norm), the $\linf$-norm, and the nuclear norm. This also includes so-called analysis-type priors, i.e. composition of the previously mentioned penalties with linear operators, typical examples being the total variation or fused Lasso penalties. 
We study the sensitivity of {\textit{any}} regularized minimizer to perturbations of the observations and provide its precise local parameterization.
Our main sensitivity analysis result shows that the predictor moves locally stably along the same active submanifold as the observations undergo small perturbations. This local stability is a consequence of the smoothness of the regularizer when restricted to the active submanifold, which in turn plays a pivotal role to get a closed form expression for the variations of the predictor w.r.t. observations. We also show that, for a variety of regularizers, including polyhedral ones or the group Lasso and its analysis counterpart, this divergence formula holds Lebesgue almost everywhere.
When the perturbation is random (with an appropriate continuous distribution), this allows us to derive an unbiased estimator of the degrees of freedom and of the risk of the estimator prediction.
Our results hold true without requiring the design matrix to be full column rank.
They generalize those already known in the literature such as the Lasso problem, the general Lasso problem (analysis $\lun$-penalty), or the group Lasso where existing results for the latter assume that the design is full column rank.

\keywords{Degrees of freedom \and Partial smoothness \and Manifold \and Sparsity \and Model selection \and o-minimal structures \and Semi-algebraic sets \and Group Lasso \and Total variation}
\end{abstract}
% Local Variables: 
% TeX-master: "../l1l2-variations-dof.tex"
% End:

%% file: sections/intro.tex
\section{Introduction}
\label{sec:introduction}

%%%%%%%%%%%%%%%%%%%%%%%%%%%%%%%%%%%%%%%%%
\subsection{Regression and Regularization}

We consider a model
\begin{equation}\label{eq:linear-problem}
  \EE(Y|X) = h(\XX\xx_0),
\end{equation}
where $Y=(Y_1,\ldots,Y_n)$ is the response vector, $\xx_0 \in \RR^p$ is the unknown vector of linear regression coefficients, $\XX \in \RR^{n \times p}$ is the fixed design matrix whose columns are the $p$ covariate vectors, and the expectation is taken with respect to some $\sigma$-finite measure. $h$ is a known real-valued and smooth function $\RR^n \to \RR^n$. The goal is to design an estimator of $\xx_0$ and to study its properties. In the sequel, we do not make any specific assumption on the number of observations $n$ with respect to the number of predictors $p$. Recall that when $n < p$,~\eqref{eq:linear-problem} is underdetermined, whereas when $n \geq p$ and all the columns of $\XX$ are linearly independent, it is overdetermined.

Many examples fall within the scope of model \eqref{eq:linear-problem}. We here review two of them.

\begin{ex}[GLM]
One naturally thinks of generalized linear models (GLMs) \citep{McCullaghNelder89} which assume that conditionally on $\XX$, $Y_i$ are independent with distribution that belongs to a given (one-parameter) standard exponential family. Recall that the random variable $Z \in \RR$ has a distribution in this family if its distribution admits a density with respect to some reference $\sigma$-finite measure on $\RR$ of the form
\[
p(z;\theta) = B(z)\exp(z\theta - \varphi(\theta)), \quad \theta \in \Theta \subseteq \RR ~,
\]
where $\Theta$ is the natural parameter space and $\theta$ is the canonical parameter. For model \eqref{eq:linear-problem}, the distribution of $Y$
belongs to the $n$-parameter exponential family and its density reads
\begin{equation}
\label{eq:pdfexp}
f(y|\XX;\xx_0) = \pa{\prod_{i=1}^n B_i(y_i)}\exp\pa{\dotp{y}{\XX\xx_0} - \sum_{i=1}^n\varphi_i\pa{\pa{\XX\xx_0}_i}}, \quad \XX\xx_0 \in \Theta^n ~,
\end{equation}
where $\dotp{\cdot}{\cdot}$ is the inner product, and the canonical parameter vector is the linear predictor $\XX\xx_0$. In this case, $h(\mu)=(h_i(\mu_i))_{1 \leq i \leq n}$, where $h_i$ is the {\textit{inverse}} of the link function in the language of GLM. Each $h_i$ is a monotonic differentiable function, and a typical choice is the canonical link $h_i=\varphi_i'$, where $\varphi_i'$ is one-to-one if the family is regular \citep{Brown86}. 
%Well-known examples are the identity link $h_i(t)=t$ (Gaussian distribution, linear model), the reciprocal link $h_i(t)=-1/t$ (Gamma and exponential distributions), and the logit link $h_i(t)=\tfrac{1}{1+\exp(-t)}$ (Bernoulli distribution, logistic regression).
\end{ex}

\begin{ex}[Transformations]
The second example is where $h$ plays the role of a transformation such as variance-stabilizing transformations (VSTs), symmetrizing transformations, or bias-corrected transformations. There is an enormous body of literature on transformations, going back to the early 1940s. A typical example is when $Y_i$ are independent Poisson random variables $\sim \Pp\pa{(\XX\xx_0)_i}$, in which case $h_i$ takes the form of the Anscombe bias-corrected VST. See \citep[Chapter~4]{DasGupta08} for a comprehensive treatment and more examples.
\end{ex}

%%%%%%%%%%%%%%%%%%%%%%%%%%%%%%%%%%%%%%%%%
\subsection{Variational Estimators}
\label{sec:block_regularizations}

Regularization is now a central theme in many fields including statistics, machine learning and inverse problems. It allows one to impose on the set of candidate solutions some prior structure on the object to be estimated. This regularization ranges from squared Euclidean or Hilbertian norms~\citep{Tikhonov97}, to non-Hilbertian norms that have sparked considerable interest in the recent years.

Given observations $(y_1,\ldots,y_n)$, we consider the class of estimators obtained by solving the convex optimization problem
\begin{equation}\label{eq-group-lasso}\tag{\regulP{y}}
  \xsoly(y) \in
  \uArgmin{ \xx \in \RR^p }
  \F(\xx,y) + \J(\xx) ~.
\end{equation}
The fidelity term $\F$ is of the following form
\begin{equation}\label{eq-fidelity-decompos}
  \F(\xx,y) = \F_0(\XX \xx, y)
\end{equation}
where $\F_0(\cdot,y)$ is a general loss function assumed to be a proper, convex and sufficiently smooth function of its first argument; see Section~\ref{sec:blocks} for a detailed exposition of the smoothness assumptions. The regularizing penalty $\J$ is proper lower semicontinuous and convex, and promotes some specific notion of simplicity/low-complexity on $\xsoly(y)$; see Section~\ref{sec:blocks} for a precise description of the class of regularizing penalties $\J$ that we consider in this paper. The type of convex optimization problem in \eqref{eq-group-lasso} is referred to as a regularized $M$-estimator in \cite{WainwrightDecomposable12}, where $\J$ is moreover assumed to have a special decomposability property. 

We now provide some illustrative examples of loss functions $F$ and regularizing penalty $J$ routinely used in signal processing, imaging sciences and statistical machine learning.

\begin{ex}[Generalized linear models]\label{exp-glm}
Generalized linear models in the exponential family falls into the class of losses we consider. Indeed, taking the negative log-likelihood corresponding to \eqref{eq:pdfexp} gives\footnote{Strictly speaking, the minimization may have to be over a convex subset of $\RR^p$.}
\begin{equation}
\label{eq:fidexp}
\F_0(\mu,y) = \sum_{i=1}^n\varphi_i\pa{\mu_i} - \dotp{y}{\mu} ~.
\end{equation}
It is well-known that if the exponential family is regular, then $\varphi_i$ is proper, infinitely differentiable, its Hessian is definite positive, and thus it is strictly convex~\citep{Brown86}. Therefore, $\F_0(\cdot,y)$ shares exactly the same properties. We recover the squared loss $\F_0(\mu,y)=\frac{1}{2}\norm{y-\mu}^2$ for the standard linear models (Gaussian case), and the logistic loss $\F_0(\mu,y) = \sum_{i=1}^n\log\pa{1+\exp(\mu_i)} - \dotp{y}{\mu}$ for logistic regression (Bernoulli case).

GLM estimators with losses \eqref{eq:fidexp} and $\lun$ or $\lun-\ldeux$ (group) penalties have been previously considered and some of their properties studied including in \citep{Bunea08,VandeGeer08,VandeGeer08b,Meier08,Bach10,Kakade10}; see also \cite[Chapter 3, 4 and 6]{BuhlmannVandeGeerBook11}.
\end{ex}

\begin{ex}[Lasso]
The Lasso regularization is used to promote the sparsity of the minimizers, see~\citep{chen1999atomi,tibshirani1996regre,osborne2000new,donoho2006most,Candes09,BickelLassoDantzig07}, and~\citep{BuhlmannVandeGeerBook11} for a comprehensive review. It corresponds to choosing $\J$ as the $\lun$-norm
\begin{equation}
  \label{lun-synthesis}
  J(\xx) = \normu{\xx} = \sum_{i=1}^p \abs{\xx_i} .
\end{equation}
It is also referred to as $\lun$-synthesis in the signal processing community, in contrast to the more general $\lun$-analysis sparsity penalty detailed below.
\end{ex}

\begin{ex}[General Lasso]
To allow for more general sparsity penalties, it may be desirable to promote sparsity through a linear operator $D = (d_1,\ldots,\allowbreak d_q) \in \RR^{p \times q}$. This leads to the so-called analysis-type sparsity penalty (a.k.a. general Lasso after~\cite{tibshirani2012dof}) where the $\lun$-norm is pre-composed by $D^*$, hence giving
\begin{equation}
  \label{lun-analysis}
  J(\xx) = \normu{D^* \xx} = \sum_{j=1}^q \abs{\dotp{d_j}{\xx}} .
\end{equation}
This of course reduces to the usual lasso penalty \eqref{lun-synthesis} when $D = \Id_p$. The penalty \eqref{lun-analysis} encapsulates several important penalties including that of the 1-D total variation~\citep{rudin1992nonlinear}, and the fused Lasso \citep{tibshirani2005sparsity}. In the former, $D^*$ is a finite difference approximation of the derivative, and in the latter, $D^*$ is the concatenation of the identity matrix $\Id_p$ and the finite difference matrix to promote both the sparsity of the vector and that of its variations.
\end{ex}

\begin{ex}[$\linf$ Anti-sparsity]
\label{ex:linf}
In some cases, the vector to be reconstructed is expected to be flat. Such a prior can be captured using the $\linf$ norm (a.k.a. Tchebycheff norm)
\begin{equation}
  \label{linf}
	J(\xx) = \normi{\xx} = \umax{i \in \ens{1,\dots,p}} \abs{\xx_i}.
\end{equation}
More generally, it is worth mentioning that a finite-valued function $\J$ is polyhedral convex (including Lasso, general Lasso, $\linf$) if and only if can be expressed as $\umax{i \in \ens{1,\dots,q}} \dotp{d_i}{\xx} - b_i$, where the vectors $d_i$ define the facets of the sublevel set at $1$ of the penalty~\citep{Rockafellar96}. The $\linf$ regularization has found applications in computer vision~\citep{jegou2012anti}, vector quantization~\citep{lyubarskii2010uncertainty}, or wireless network optimization~\citep{studer12signal}.
\end{ex}

\begin{ex}[Group Lasso]
When the covariates are assumed to be clustered in a few active groups/blocks, the group Lasso has been advocated since it promotes sparsity of the groups, i.e. it drives all the coefficients in one group to zero together hence leading to group selection, see~\citep{bakin1999adaptive,yuan2006model,bach2008consistency,Wei10} to cite a few. The group Lasso penalty reads
\begin{equation}
  \label{lun-deux-synthesis}
  J(\xx) = \norm{\xx}_{1,2} = \sum_{b \in \Bb} \norm{\xx_b}_2 .
\end{equation}
where $\xx_b=(\xx_i)_{i \in b}$ is the sub-vector of $\xx$ whose entries are indexed by the block $b \in \Bb$ where $\Bb$ is a disjoint union of the set of indices i.e.~$\bigcup_{b \in \Bb} = \ens{1,\ldots,p}$ such that $b, b' \in \Bb, b \cap b' = \emptyset$. The mixed $\lun-\ldeux$ norm defined in~\eqref{lun-deux-synthesis} has the attractive property to be invariant under (groupwise) orthogonal transformations.
\end{ex}

\begin{ex}[General Group Lasso]
One can push the structured sparsity idea one step further by promoting group/block sparsity through a linear operator, i.e. analysis-type group sparsity. Given a collection of linear operators $\{D_b\}_{b \in \Bb}$, that are not all orthogonal, the analysis group sparsity penalty is 
\begin{equation}
  \label{lun-deux-analysis}
  J(\xx) = \norm{D^* \xx}_{1,2} = \sum_{b \in \Bb} \norm{D_b^* \xx}_2. 
\end{equation}
This encompasses the 2-D isotropic total variation~\citep{rudin1992nonlinear}, where $\xx$ is a 2-D discretized image, and each $D_b^* \xx \in \RR^2$ is a finite difference approximation of the gradient of $\xx$ at a pixel indexed by $b$. The overlapping group Lasso \citep{jacob-overlap-synthesis} is also a special case of \eqref{lun-deux-analysis} by taking $D_b^* : \xx \mapsto \xx_b$ to be a block extractor operator~\citep{peyre2011adaptive,chen-proximal-overlap}.
\end{ex}

\begin{ex}[Nuclear norm]
The natural extension of low-complexity priors to matrix-valued objects $\xx \in \RR^{p_1 \times p_2}$ (where $p=p_1p_2$) is to penalize the singular values of the matrix. Let $U_{\xx} \in \RR^{p_1 \times p_1}$ and $V_{\xx} \in \RR^{p_2 \times p_2}$ be the orthonormal matrices of left and right singular vectors of $\xx$, and $\uplambda: \RR^{p_1 \times p_2} \to \RR^{p_2}$ is the mapping that returns the singular values of $\xx$ in non-increasing order. If $j \in \lsc(\RR^{p_2})$, i.e. convex, lower semi-continuous and proper, is an absolutely permutation-invariant function, then one can consider the penalty $J(\xx) = j(\uplambda(\xx))$. This is a so-called spectral function, and moreover, it can be also shown that $J \in \lsc(\RR^{p_1 \times p_2})$~\citep{LewisMathEig}. The most popular spectral penalty is the nuclear norm obtained for $j=\normu{\cdot}$, 
\begin{equation}
  \label{eq-nuclear-norm} 
  J(\xx) = \norm{\xx}_* = \normu{\uplambda(\xx)} ~.
\end{equation}
This penalty is the best convex candidate to enforce a low-rank prior. It has been widely used for various applications, including low rank matrix completion~\citep{recht2010guaranteed,candes2009exact}, robust PCA~\citep{CandesRPCA11}, model reduction~\citep{fazel2001rank}, and phase retrieval~\citep{CandesPhaseLift}. 
\end{ex}

%%%%%%%%%%%%%%%%%%%%%%%%%%%%%%%%%%%%%%%%%
\subsection{Sensitivity Analysis }
\label{sub:intro-sensitivity}

A chief goal of this paper is to investigate the sensitivity of any solution $\xsoly(y)$ to the parameterized problem~\eqref{eq-group-lasso} to (small) perturbations of $y$. Sensitivity analysis\footnote{The meaning of sensitivity is different here from what is usually intended in statistical sensitivity and uncertainty analysis.} is a major branch of optimization and optimal control theory. Comprehensive monographs on the subject are \citep{BonnansShapiro2000,mordukhovich1992sensitivity}. The focus of sensitivity analysis is the dependence and the regularity properties of the optimal solution set and the optimal values when the auxiliary parameters (e.g. $y$ here) undergo a perturbation. In its simplest form, sensitivity analysis of first-order optimality conditions, in the parametric form of the Fermat rule, relies on the celebrated implicit function theorem.

The set of regularizers $J$ we consider is that of partly smooth functions relative to a Riemannian submanifold as detailed in Section~\ref{sec:blocks}. The notion of partial smoothness was introduced in~\citep{Lewis-PartlySmooth}. This concept, as well as that of identifiable surfaces~\citep{Wright-IdentSurf}, captures essential features of the geometry of non-smoothness which are along the so-called ``active/identifiable manifold''. For convex functions, a closely related idea was developed in \citep{Lemarechal-ULagrangian}. Loosely speaking, a partly smooth function behaves smoothly as we move on the identifiable manifold, and sharply if we move normal to the manifold. In fact, the behaviour of the function and of its minimizers (or critical points) depend essentially on its restriction to this manifold, hence offering a powerful framework for sensitivity analysis theory. In particular, critical points of partly smooth functions move stably on the manifold as the function undergoes small perturbations~\citep{Lewis-PartlySmooth,Lewis-PartlyTiltHessian}. 

Getting back to our class of regularizers, the core of our proof strategy relies on the identification of the active manifold associated to a particular minimizer $\xsoly(y)$ of~\eqref{eq-group-lasso}. We exhibit explicitly a certain set of observations, denoted $\Hh$ (see Definition~\ref{defn:h}), outside which the initial non-smooth optimization~\eqref{eq-group-lasso} boils down locally to a smooth optimization along the active manifold. This part of the proof strategy is in close agreement with the one developed in~\citep{Lewis-PartlySmooth} for the sensitivity analysis of partly smooth functions. See also \citep[Theorem~13]{Bolte2011} for the case of linear optimization over a convex semialgebraic partly smooth feasible set, where the authors proves a sensitivity result with a zero-measure transition space. However, it is important to stress that neither the results of \citep{Lewis-PartlySmooth} nor those of \citep{Bolte2011,LewisGenericConvex11} can be applied straightforwardly in our context for two main reasons (see also Remark~\ref{rem:transpace} for a detailed discussion). In all these works, a non-degeneracy assumption is crucial while it does not necessarily hold in our case, and this is precisely the reason we consider the boundary of the sets $\Hh_{\Mm}$ in the definition of the transition set $\Hh$. Moreover, in the latter papers, the authors are concerned with a particular type of perturbations (see Remark~\ref{rem:transpace}) which does not allow to cover our class of regularized problems except for restrictive cases such as $\XX$ injective. For our class of problems \lasso, we were able to go beyond these works by solving additional key challenges that are important in a statistical context, namely: (i) we provide an analytical description of the set $\Hh$ involving the boundary of $\Hh_{\Mm}$, which entails that $\Hh$ is potentially of dimension strictly less than $n$, hence of zero Lebesgue measure, as we will show under a mild o-minimality assumption. (ii) we prove a general sensitivity analysis result valid for any proper lower semicontinuous convex partly smooth regularizer $\J$; (iii) we compute the first-order expansion of $\xsoly(y)$ and provide an analytical form of the weak derivative of $y \mapsto \XX\xsoly(y)$ valid outside a set involving $\Hh$. If this set is of zero-Lebesgue measure, this allows us to get an unbiased estimator of the risk on the prediction $\XX\xsoly(Y)$.

%%%%%%%%%%%%%%%%%%%%%%%%%%%%%%%%%%%%%%%%%
\subsection{Degrees of Freedom and Unbiased Risk Estimation}
\label{sub:intro-risk}

The degrees of freedom (DOF) of an estimator quantifies the complexity of a statistical modeling procedure~\citep{efron1986biased}. It is at the heart of several risk estimation procedures and thus allows one to perform parameter selection through risk minimization. 

In this section, we will assume that $\F_0$ in \eqref{eq-fidelity-decompos} is strictly convex, so that the response (or the prediction) $\msol(y)=\XX \xsoly(y)$ is uniquely defined as a single-valued mapping of $y$ (see Lemma~\ref{lem:unique}). That is, it does not depend on a particular choice of solution $\xsoly(y)$ of~\eqref{eq-group-lasso}. 

Let $\mu_0 = \XX \xx_0$. Suppose that $h$ in \eqref{eq:linear-problem} is the identity and that the observations $Y \sim \Nn(\mu_0,\sigma^2 \Id_n)$. Following~\citep{efron1986biased}, the DOF is defined as
\eq{
	\DOF = \sum_{i=1}^n \frac{\mathrm{cov}(Y_i, \msol_i(Y))}{\sigma^2} ~.
}
The well-known Stein's lemma~\citep{stein1981estimation} asserts that, if $y \mapsto \msol(y)$ is weakly differentiable function (i.e. typically in a Sobolev space over an open subset of $\RR^n$), such that each coordinate $y \mapsto \msol_i(y) \in \RR$  has an essentially bounded weak derivative\footnote{We write the same symbol as for the derivative, and rigorously speaking, this has to be understood to hold Lebesgue-a.e.} 
\eq{
\EE\pa{\Big| \frac{\partial \msol_i}{\partial y_i}(Y)\Big|} < \infty,  \quad \forall i ~,
}
then its divergence is an unbiased estimator of its DOF, i.e.
\eq{
  \widehat{\DOF} = \diverg(\msol)(Y) \eqdef \tr(\jac \msol(Y)) \qandq \EE(\widehat{\DOF}) = \DOF ~,
}
where $\jac \msol$ is the Jacobian of $y \mapsto \msol(y)$. In turn, this allows to get an unbiased estimator of the prediction risk $\EE( \norm{\msol(Y) - \mu_0}^2)$ through the SURE~\citep{stein1981estimation}.

Extensions of the SURE to independent variables from an exponential family are considered in \citep{hudson1978nie} for the continuous case, and \citep{Hwang82} in the discrete case. \cite{eldar-gsure} generalizes the SURE principle to continuous multivariate exponential families.

%%%%%%%%%%%%%%%%%%%%%%%%%%%%%%%%%%%%%%%%%
\subsection{Contributions} % (fold)
\label{sub:intro-contrib}

We consider a large class of losses $\F_0$, and of regularizing penalties $\J$ which are proper, lower semicontinuous, convex and partly smooth functions relative to a Riemannian submanifold, see Section~\ref{sec:blocks}.
For this class of regularizers and losses, we first establish in Theorem~\ref{thm-local} a general sensitivity analysis result, which provides the local parametrization of any solution to~\eqref{eq-group-lasso} as a function of the observation vector $y$. This is achieved without placing any specific assumption on $\XX$, should it be full column rank or not. We then derive an expression of the divergence of the prediction with respect to the observations (Theorem~\ref{thm-div}) which is valid outside a set of the form $\Gg \cap \Hh$, where $\Gg$ is defined in Section~\ref{sec:sensmu}. Using tools from o-minimal geometry, we prove that the transition set $\Hh$ is of Lebesgue measure zero. If $\Gg$ is also negligible, then the divergence formula is valid Lebesgue-a.e.. In turn, this allows us to get an unbiased estimate of the DOF and of the prediction risk (Theorem~\ref{thm-dof} and Theorem~\ref{thm-dof-exp}) for model \eqref{eq:linear-problem} under two scenarios: (i) Lipschitz continuous non-linearity $h$ and an additive i.i.d. Gaussian noise; (ii) GLMs with a continuous exponential family. Our results encompass many previous ones in the literature as special cases (see discussion in the next section). It is important however to mention that though our sensitivity analysis covers the case of the nuclear norm (also known as the trace norm), unbiasedness of the DOF and risk estimates is not guaranteed in general for this regularizer as the restricted positive definiteness assumption (see Section~\ref{sec:local}) may not hold at any minimizer (see Example~\ref{ex:injnuc}), and thus $\Gg$ may not be always negligible.
%Though this penalty is finite-valued, convex and partly smooth, its partial smoothness manifold is not affine nor linear; it is in fact the manifold of constant rank matrices. Moreover, the set of tangent spaces is not finite (see~Section~\ref{sec:dof}). 

%%%%%%%%%%%%%%%%%%%%%%%%%%%%%%%%%%%%%%%%%
\subsection{Relation to prior works}

In the case of standard Lasso (i.e. $\lun$ penalty \eqref{lun-synthesis}) with $Y \sim \Nn(\XX\xx_0,\sigma^2\Id_n)$ and $\XX$ of full column rank, ~\citep{zou2007degrees} showed that the number of nonzero coefficients is an unbiased estimate for the DOF. Their work was generalized in~\citep{2012-kachour-statsinica} to any arbitrary design matrix. Under the same Gaussian linear regression model, unbiased estimators of the DOF for the general Lasso penalty~\eqref{lun-analysis}, were given independently in~\citep{tibshirani2012dof,vaiter-local-behavior}.

A formula of an estimate of the DOF for the group Lasso when the design is orthogonal within each group was conjectured in~\citep{yuan2006model}. \cite{kato2009degrees} studied the DOF of a general shrinkage estimator where the regression coefficients are constrained to a closed convex set $\Cc$. His work extends that of~\citep{MeyerWoodroofe} which treats the case where $\Cc$ is a convex polyhedral cone. When $\XX$ is full column rank,~\citep{kato2009degrees} derived a divergence formula under a smoothness condition on the boundary of $\Cc$, from which an unbiased estimator of the degrees of freedom was obtained.
When specializing to the constrained version of the group Lasso, the author provided an unbiased estimate of the corresponding DOF under the same group-wise orthogonality assumption on $\XX$ as~\citep{yuan2006model}.
\cite{hansen2014dof} studied the DOF of the metric projection onto a closed set (non-necessarily convex), and gave a precise representation of the bias when the projector is not sufficiently differentiable.
An estimate of the DOF for the group Lasso was also given by~\citep{solo2010threshold} using heuristic derivations that are valid only when $\XX$ is full column rank, though its unbiasedness is not proved. 

\cite{vaiter-icml-workshops} also derived an estimator of the DOF of the group Lasso and proved its unbiasedness when $\XX$ is full column rank, but without the orthogonality assumption required in~\citep{yuan2006model,kato2009degrees}. When specialized to the group Lasso penalty, our results establish that the DOF estimator formula in~\citep{vaiter-icml-workshops} is still valid while removing the full column rank assumption. This of course allows one to tackle the more challenging rank-deficient or underdetermined case $p>n$.

% Local Variables: 
% TeX-master: "../l1l2-variations-dof.tex"
% End:

%% file: sections/notations.tex
%%%%%%%%%%%%%%%%%%%%%%%%%%%%%%%%%%%%%%%%%%
\section{Notations and preliminaries}
\label{sec:notations}

\myparagraph{Vectors and matrices}
Given a non-empty closed set $\Cc \subset \RR^p$, we denote $\proj_{\Cc}$ the orthogonal projection on $\Cc$. For a subspace $\T \subset \RR^p$, we denote
\eq{
	\xx_{\T} = \proj_{\T}\xx \qandq
	\XXT = \XX \proj_{\T}.
}
For a set of indices $I \subset \NN^*$, we will denote $\xx_I$ (resp. $\XX_I$) the subvector (resp. submatrix) whose entries (resp. columns) are those of $\xx$ (resp. of $\XX$) indexed by $I$. For a linear operator $A$, $\adj{A}$ is its adjoint. For a matrix $M$, $\transp{M}$ is its transpose and $M^+$ its Moore-Penrose pseudo-inverse.

\myparagraph{Sets} 
In the following, for a non-empty set $\Cc \subset \RR^p$, we denote $\co \Cc$ and $\cone \Cc$ respectively its convex and conical hulls. $\iota_{\Cc}$ is the indicator function of $\Cc$ (takes $0$ in $\Cc$ and $+\infty$ otherwise), and $N_\Cc(\xx)$ is the cone normal to $\Cc$ at $\xx$. For a non-empty convex set $\Cc$, its affine hull $\Aff C$ is the smallest affine manifold containing it. It is a translate of $\Lin \Cc$, the subspace parallel to $\Cc$, i.e. $\Lin \Cc = \Aff \Cc - \xx = \RR(\Cc - \Cc)$ for any $\xx \in \Cc$. The relative interior $\ri \Cc$ (resp. relative boundary $\rbd \Cc$) of $\Cc$ is its interior (resp. boundary) for the topology relative to its affine hull.

\myparagraph{Functions} 
For a $\Calt{1}$ vector field $v: y \in \RR^n \mapsto v(y)$, $\jac v(y)$ denotes its Jacobian at $y$. 
For a $\Cdeux$ smooth function $\tilde{f}$, $\dder \tilde{f}(\xx)[\xi]=\dotp{\grad \tilde{f}(\xx)}{\xi}$ is its directional derivative, $\grad \tilde{f}(\xx)$ is its (Euclidean) gradient and $\hess \tilde{f}(\xx)$ is its (Euclidean) Hessian at $\xx$. For a bivariate function $g: (\xx,y) \in \RR^p \times \RR^n \to \RR$ that is $\Cdeux$ with respect to the first variable $\xx$, for any $y$, we will denote $\grad g(\xx,y)$ and $\hess g(\xx,y)$ the  gradient and Hessian of $g$ at $\xx$ with respect to the first variable.

A function $f: \xx \in \RR^p \mapsto \RR \cup \ens{+\infty}$ is lower semicontinuous (lsc) if its epigraph is closed. $\lsc(\RR^p)$ is the class of convex and lsc functions which are proper (i.e.~not everywhere $+\infty$). $\partial f$ is the (set-valued) subdifferential operator of $f \in \lsc(\RR^p)$. If $f$ is differentiable at $\xx$ then $\grad f(\xx)$ is its unique subgradient, i.e. $\partial f(\xx) = \ens{\grad f(\xx)}$. 

Consider a function $\J \in \lsc(\RR^p)$ such that $\partial \J(\xx) \neq \emptyset$. We denote $\S_\xx$ the subspace parallel to $\partial J(\xx)$ and its orthogonal complement $\T_\xx$, i.e.
\begin{equation}\label{eq-dfn-T-gauges}
  \S_\xx = \Lin(\partial J(\xx)) \qandq \T_\xx = \S_\xx^\perp.
\end{equation}
We also use the notation 
\begin{equation*}
  \e{\xx} = \proj_{\Aff(\partial J(\xx))}(0) ,
\end{equation*}
i.e. the projection of 0 onto the affine hull of $\partial \J(\xx)$.

\myparagraph{Differential and Riemannian geometry}
Let $\Mm$ be a $\Cdeux$-smooth embedded submanifold of $\RR^p$ around $\xx^\star \in \Mm$. To lighten notation, henceforth we shall state $\Cdeux$-manifold instead of $\Cdeux$-smooth embedded submanifold of $\RR^p$. $\Tgt_{\xx}(\Mm)$ denotes the tangent space to $\Mm$ at any point $\xx \in \Mm$ near $\xx^\star$. The natural embedding of a submanifold $\Mm$ into $\RR^p$ permits to define a Riemannian structure on $\Mm$, and we simply say $\Mm$ is a Riemannian manifold. For a vector $v \in \Tgt_{\xx}(\Mm)^\perp$, the Weingarten map of $\Mm$ at $\xx$ is the operator $\Afk_{\xx}(\cdot,v): \Tgt_{\xx}(\Mm) \to \Tgt_{\xx}(\Mm)$ defined as
\[
\Afk_{\xx}(\xi,v) = -\proj_{\Tgt_{\xx}(\Mm) }\dder V[\xi]
\]
where $V$ is any local extension of $v$ to a normal vector field on $\Mm$. The definition is independent of the choice of the extension $V$, and $\Afk_{\xx}(\cdot,v)$ is a symmetric linear operator which is closely tied to the second fundamental form of $\Mm$; see~\citep[Proposition~II.2.1]{chavel2006smooth}.

Let $f$ be a real-valued function which is $\Cdeux$ on $\Mm$ around $\xx^\star$. The covariant gradient of $f$ at $\xx$ is the vector $\grad_{\Mm} f(\xx) \in \Tgt_\xx(\Mm)$ such that
\[
\dotp{\grad_{\Mm} f(\xx)}{\xi} = \frac{d}{dt}f\pa{\proj_{\Mm}(\xx+t\xi)}\big|_{t=0} , \forall \xi \in \Tgt_\xx(\Mm) ~.
\]
The covariant Hessian of $f$ at $\xx$ is the symmetric linear mapping $\hess_\Mm f(\xx)$ from $\Tgt_\xx(\Mm)$ into itself defined as
\[
\dotp{\hess_{\Mm} f(\xx)\xi}{\xi} = \frac{d^2}{dt^2}f\pa{\proj_{\Mm}(\xx+t\xi)}\big|_{t=0} , \forall \xi \in \Tgt_\xx(\Mm) ~.
\]
This definition agrees with the usual definition using geodesics or connections \citep{miller2005newton}. Assume now that $\Mm$ is a Riemannian embedded submanifold of $\RR^p$, and that a function $f$ has a smooth restriction on $\Mm$. This can be characterized by the existence of a smooth extension (representative) of $f$, i.e.~a smooth function $\tilde{f}$ on $\RR^p$ such that $\tilde{f}$ and $f$ agree on $\Mm$. Thus, the Riemannian gradient $\grad_{\Mm} f(\xx)$ is also given by
\eql{
\label{eq:covgrad}
\grad_{\Mm} f(\xx) = \proj_{\Tgt_\xx(\Mm)} \grad \tilde{f}(\xx)
}
and, $\forall \xi \in \Tgt_\xx(\Mm)$, the Riemannian Hessian reads
\begin{align}
\label{eq:covhess}
\hess_\Mm f(\xx)\xi	&= \proj_{\Tgt_\xx(\Mm)} \dder\pa{\grad_{\Mm} f}(\xx)[\xi] = \proj_{\Tgt_\xx(\Mm)}\dder\pa{\xx \mapsto \proj_{\Tgt_\xx(\Mm)} \grad \tilde{f}(\xx)}[\xi] \nonumber\\
			&= \proj_{\Tgt_\xx(\Mm)} \hess \tilde{f}(\xx) \proj_{\Tgt_\xx(\Mm)} \xi + \Afk_{\xx}\pa{\xi,\proj_{\Tgt_\xx(\Mm)^\perp} \grad \tilde{f}(\xx)} ~,
\end{align}
where the last equality comes from \citep[Theorem~1]{Absil13}. When $\Mm$ is an affine or linear subspace of $\RR^p$, then obviously $\Mm=\xx+\Tgt_\xx(\Mm)$, and $\Afk_{\xx}\pa{\xi,\proj_{\Tgt_\xx(\Mm)^\perp} \tilde{f}(\xx)}=0$, hence \eqref{eq:covhess} becomes
\eql{
\label{eq:covhesslin}
\hess_{\Mm} f(\xx) = \proj_{\Tgt_\xx(\Mm)}\hess \tilde{f}(\xx)\proj_{\Tgt_\xx(\Mm)}  ~.
}
Similarly to the Euclidean case, for a real-valued bivariate function $g$ that is $\Cdeux$ on $\Mm$ around the first variable $\xx$, for any $y$, we will denote $\grad_\Mm g(\xx,y)$ and $\hess_\Mm g(\xx,y)$ the Riemannian gradient and Hessian of $g$ at $\xx$ with respect to the first variable. See e.g. \citep{lee2003smooth,chavel2006smooth} for more material on differential and Riemannian manifolds.

%For a subspace $\T \subset \RR^p$, and any function $g \in \Cdeux(T \times \RR^n)$, we denote
%\eq{\label{defn-hessianF}
%	\jac_1^2 g_\T(\xx,y) 
%        =
%        \proj_\T \circ \, \jac_1^2 g(\xx,y) \circ \proj_\T
%}
%which can be understood as the Hessian of the mapping $\xx \in \T \mapsto g(\xx,y)$, i.e. the restriction of $g(\cdot,y)$ to $\T$.
%Of course, when $\T$ is the whole space, we recover the ``full'' Hessian.
%
%We also denote $\jac_{12}^2 g(\xx,y)$ the Jacobian of the mapping $y \in \RR^n \mapsto \grad_1 g(\xx,y)$ with respect to $y$, and $\grad_1 g(\xx,y)$ is the gradient of $g$ w.r.t the first variable at $(\xx,y)$.

%We now turn to the notion of \emph{model space}.
%The interested reader may refer to \citep{vaiter2013model} and references therein for a comprehensive treatment.

%We denote the set of all possible subspaces $\T_{\xx}$ as
%\begin{equation}
%\label{eq:Tt}
%  \Tt = \bigcup_{\xx \in \RR^p}\T_\xx.
%\end{equation}
%For any $\T \in \Tt$, we denote $\Tc$ the set of vectors sharing the same subspace $\T$,
%\begin{equation*}
%  \Tc = \enscond{\xx' \in \RR^p}{\T_{\xx'}=\T}.
%\end{equation*}
%For instance, when $J=\norm{\cdot}_1$, $\Tc_\xx$ is the cone of all vectors sharing the same support as $\xx$.

%% file: sections/prgauges.tex
\section{Partly Smooth Functions} % (fold)
\label{sec:blocks}

%%%%%%%%%%%%%%%%%%%%%%%%%%%%%%%%%%%%%%%%%%%%%%%%%%%%%%
\subsection{Partial Smoothness}

Toward the goal of studying the sensitivity behaviour of $\widehat{\xx}(y)$ and $\msol(y)$ with regularizers $\J \in \lsc(\RR^p)$, we restrict our attention to a subclass of these functions that fulfill some regularity assumptions according to the following definition.
\begin{defn}\label{defn:psg}
  Let $J \in \lsc(\RR^{p})$ and a point $\xx$ such that $\partial J(\xx) \neq \emptyset$. $J$ is said to be \emph{partly smooth} at $\xx$ relative to a set $\Mm \subseteq \RR^p$ if
  \begin{enumerate}
  \item Smoothness: $\Mm$ is a $\Cdeux$-manifold and $J$ restricted to $\Mm$ is $\Cdeux$ around $\xx$.
%\begin{align*}
%\label{hyp-j-reg}\tag{$C_{\text{sm}}$}
%\text{$J$ restricted to $\Mm$ is $\Cdeux$ around $\xx$.}
%\end{align*}
  \item Sharpness:
    $\Tgt_{\xx}(\Mm) = T_{\xx} \eqdef \Lin(\partial J(\xx))^\perp$.
  \item Continuity: The set-valued mapping $\partial J$ is continuous at $\xx$ relative to $\Mm$.
%\begin{align*}
%\label{hyp-j-cont}\tag{$C_{\text{cont}}$}
%\text{The set-valued mapping $\partial J$ is continuous at $\xx$ relative to $\Mm$.}
%\end{align*}
  \end{enumerate}
  $\J$ is said to be \emph{partly smooth relative to the manifold $\Mm$} if $J$ is partly smooth at each point $\xx \in \Mm$ relative to $\Mm$.
\end{defn}
Observe that $\Mm$ being affine or linear is equivalent to $\Mm=\xx+T_{\xx}$. A closed convex set $\Cc$ is partly smooth at a point $\xx \in \Cc$ relative to a $\Cdeux$-manifold $\Mm$ locally contained in $\Cc$ if its indicator function $\iota_\Cc$ maintains this property.\\

\citet[Proposition~2.10]{Lewis-PartlySmooth} allows to prove the following fact (known as local normal sharpness).

\begin{fact}
\label{fact:sharp}
If $\J$ is partly smooth at $\xx$ relative to $\Mm$, then all $\xx' \in \Mm$ near $\xx$ satisfy
\[
\Tgt_{\xx'}(\Mm) = \T_{\xx'} ~.
\]
In particular, when $\Mm$ is affine or linear, then
\begin{align*}
%\label{hyp-j-prg} \tag{$C_{\mathrm{sharp}}$}
\text{$\forall \xx' \in \Mm$ near $\xx$,} \quad \T_{\xx'} = \T_{\xx} ~.
\end{align*}
\end{fact}

%Some remarks are in order. Assumption~\eqref{hyp-j-prg} amounts to saying that there exists a neighbourhood of $\xx$ on $\T_\xx$ on which this subspace model is constant. 
It can also be shown that the class of partly smooth functions enjoys a powerful calculus. For instance, under mild conditions, it is closed under positive combination, pre-composition by a linear operator and spectral lifting, with closed-form expressions of the resulting partial smoothness manifolds and their tangent spaces, see~\citep{Lewis-PartlySmooth,vaiter2014model}. 
%Many well-studied regularizing penalties are partly smooth relative to a linear manifold, including the $\lun$, the $\lun-\ldeux$, the $\linf$ norms, and their analysis-type versions and/or positive combinations, see~\citep[Section~6]{vaiter2013model} for a detailed discussion of the examples.

It turns out that except the nuclear norm, the regularizing penalties that we exemplified in Section~\ref{sec:introduction} are partly smooth relative to a linear subspace. The nuclear norm is partly smooth relative to the fixed-rank manifold.

%%%%%
\begin{ex}[Lasso]
\label{ex:lassoM}
We denote $(a_i)_{1 \leq i \leq p}$ the canonical basis of $\RR^p$. Then, $J=\norm{\cdot}_1$ is partly smooth at $\xx$ relative to
\begin{equation*}
  \Mm = \T_\xx = \Span \ens{(a_i)_{i \in \supp(\xx)}}
  \qwhereq
  \supp(\xx) \eqdef \enscond{i \in \ens{1,\dots,p}}{\xx_i \neq 0} .
\end{equation*}
\end{ex}

%%%%%
\begin{ex}[General Lasso]
\label{ex:analassoM}
\citet[Proposition~9]{vaiter2013model} relates the partial smoothness subspace associated to a convex partly smooth regularizer $\J \circ D^*$, where $D$ is a linear operator, to that of $\J$. In particular, for $\J=\norm{\cdot}_1$, $\J \circ D^*$ is partly smooth at $\xx$ relative to
\begin{equation*}
  \Mm = \T_\xx = \Ker(D_{\Lambda^c}^*) \qwhereq \Lambda = \supp(D^* \xx) .
\end{equation*}
\end{ex}

%%%%%
\begin{ex}[$\linf$ Anti-sparsity]
\label{ex:linfM}
It can be readily checked that $\J=\normi{\cdot}$ is partly smooth at $\xx$ relative to
\begin{equation*}
 	\Mm = \T_\xx = \enscond{\xx'}{\xx'_I \in \RR \sign(\xx_I)} 
	\qwhereq
	I = \enscond{i}{\xx_i = \normi{\xx}} ~.
\end{equation*}
\end{ex}

%%%%%
\begin{ex}[Group Lasso]
\label{ex:glassoM}
The partial smoothness subspace associated to $\xx$ when the blocks are of size greater than 1 can be defined similarly, but using the notion of block support.
Using the block structure $\Bb$, one has that the group Lasso regularizer is partly smooth at $\xx$ relative to
\begin{equation*}
  \Mm = \T_\xx = \Span \ens{(a_i)_{i \in \bs(\xx)}},
\end{equation*}
where
\begin{equation*}
  \bs(\xx) = \enscond{i \in \ens{1,\dots,p}}{\exists b \in \Bb,\, \xx_b \neq 0 \qandq i \in b} .
\end{equation*}
\end{ex}

%%%%%
\begin{ex}[General Group Lasso]
\label{ex:gglassoM}
Using again \citep[Proposition~9]{vaiter2013model}, we can describe the partial smoothness subspace for $\J=\norm{D^* \cdot}_\Bb$, which reads
\begin{equation*}
  \Mm = \T_\xx = \Ker(D_{\Lambda^c}^*) \qwhereq \Lambda = \bs(D^* \xx) .
\end{equation*}
\end{ex}

%%%%%
\begin{ex}[Nuclear norm]
\label{ex:nucM}
Piecing together \citep[Theorem~3.19]{daniilidis2013orthogonal} and Example~\ref{ex:lassoM}, the nuclear norm can be shown to be partly smooth at $\xx \in \RR^{p_1 \times p_2}$ relative to the set
\begin{equation*}
	\Mm = \enscond{\xx'}{ \rank(\xx')=r }, \quad r=\rank(\xx) ,
\end{equation*}
which is a $\Cdeux$-manifold around $\xx$ of dimension $(p_1+p_2-r)r$; see \citep[Example~8.14]{lee2003smooth}.
\end{ex}

%%%%%
\begin{ex}[Indicator function of a partly smooth set $\Cc$]
\label{ex:setM}
Let $\Cc$ be a closed convex and partly smooth set at $\xx \in \Cc$ relative to $\Mm$. Observe that when $\xx \in \ri \Cc$, $\Mm = \RR^p$. For $\xx \in \rbd \Cc$, $\Mm$ is locally contained in $\rbd \Cc$.
\end{ex}

We now consider an instructive example of a partly smooth function relative to a non-flat active submanifold that will serve as a useful illustration in the rest of the paper.
%%%%%
\begin{ex}[$\J=\max(\norm{\cdot}-1,0)$]
\label{ex:maxnormM}
We have $\J \in \lsc(\RR^p)$ and continuous. It is then differentiable Lebesgue-a.e., except on the unit sphere $\sph^{p-1}$. For $\xx$ outside $\sph^{p-1}$, $\J$ is parly smooth at $\xx$ relative to $\RR^p$. For $\xx \in \sph^{p-1}$, $\J$ is partly smooth at $\xx$ relative to $\sph^{p-1}$. Obviously, $\sph^{p-1}$ is a $\Cdeux$-smooth manifold.
\end{ex}

%%%%%%%%%%%%%%%%%%%%%%%%%%%%%%%%%%%%%%%%%%%%%%%%%%%%%%%%%%%%%%%%
\subsection{Riemannian Gradient and Hessian}
\label{subsec-restriction}
We now give expressions of the Riemannian gradient and Hessian for the case of partly smooth functions relative to a $\Cdeux$-manifold. This is summarized in the following fact which follows by combining \eqref{eq:covgrad}, \eqref{eq:covhess}, Definition~\ref{defn:psg} and \citet[Proposition~17]{Daniilidis06}.

\begin{fact}
\label{fact:gradhess}
If $\J$ is partly smooth relative at $\xx$ relative to $\Mm$, then for any $\xx' \in \Mm$ near $\xx$
\[
\JxM(\xx') = \proj_{\T_{\xx'}}\pa{\partial J(\xx')} = \e{\xx'} ~,
\]
and this does not depend on the smooth representation $\tilde{J}$ of $\J$ on $\Mm$. In turn,
\[
\Q(\xx) = \proj_{\T_{\xx}} \hess \tilde{J}(\xx) \proj_{\T_\xx} + \Afk(\cdot,\proj_{\S_\xx}\grad \tilde{J}(\xx)) ~.
\]
\end{fact}

Let's now exemplify this fact by providing the expressions of the Riemannian Hessian for the examples discussed above.

%We denote 
%\begin{equation*}
%  \J_\T : \xx_\T \in \T \mapsto \J(\xx_\T) \in \RR^+
%\end{equation*}
%the restriction of $\J$ to $\T$ for some subspace $\T \subset \Tt$.
%Hence the Hessian of $\J_\T$ is well-defined on $\Tt$.
%We illustrate this definition on several examples.

%%%%%
\begin{ex}[Polyhedral penalty]
\label{ex:lassohess}
Polyhedrality of $J$ implies that it is affine nearby $\xx$ along the partial smoothness subspace $\Mm=\xx+\T_\xx$, and its subdifferential is locally constant nearby $\xx$ along $\Mm$. In turn, the Riemannian Hessian of $J$ vanishes locally, i.e. $\Q(\xx') = 0$ for all $\xx' \in \Mm$ near $\xx$. Of course, this holds for the Lasso, general Lasso and $\linf$ anti-sparsity penalties since they are all polyhedral. 
%\eq{
%	\foralls \xx_\T \in \T, \quad
%	\grad \J_\T(\xx_\T) = \sign(\xx_\T) ~,
%}
%and thus, $\Q_T(\xx_T) = 0$.
%This is also the case for the analysis $\lun$-penalty (general Lasso), see for instance \citep{vaiter-local-behavior}. This property basically reflects the fact that these regularizers are polyhedral, hence piecewise affine. 
\end{ex}

%%%%%
\begin{ex}[Group Lasso]
\label{ex:glassohess}
Using the expression of $\Mm=T_\xx$ in Example~\ref{ex:glassoM}, it is straightforward to show that 
\begin{equation*}
  \Q(\xx) = \delta_\xx \circ Q_{\xx^\perp},
\end{equation*}
where, for $\Lambda = \bs(\xx)$, 
\begin{gather*}
  \delta_\xx : \T_{\xx} \to \T_{\xx}, v \mapsto 
  \begin{cases}
  v_b / \norm{\xx_b} & \text{if}~ \xx_b \neq 0 \\
  0 & \text{otherwise}
  \end{cases}
  \\ \qandq \\ 
  Q_{\xx^\perp} : \T_{\xx} \to \T_{\xx}, v \mapsto 
  \begin{cases}
  v_b - \frac{\dotp{\xx_b}{v_b}}{\norm{\xx_b}^2} \xx_b & \text{if}~ \xx_b \neq 0 \\
  0 & \text{otherwise}
  \end{cases} ~.
\end{gather*}
\end{ex}

%%%%%
\begin{ex}[General Group Lasso]
\label{ex:gglassohess}

Applying the chain rule to Example~\ref{ex:glassohess}, we get 
\begin{equation*}
  \Q(\xx) = \proj_{\Ker(D_{\Lambda^c}^*)} D \pa{\delta_{D^* \xx} \circ Q_{(D^* \xx)^\perp}}D^*\proj_{\Ker(D_{\Lambda^c}^*)},
\end{equation*}
where $\Lambda = \bs(D^* \xx)$ and the operator $\delta_{D^* \xx} \circ Q_{(D^* \xx)^\perp}$ is defined similarly to Example~\ref{ex:glassohess}.
\end{ex}

%%%%%
\begin{ex}[Nuclear norm]
\label{ex:nuchess}
For $\xx \in \RR^{p_1 \times p_2}$ with $\rank(\xx)=r$, let $\xx = U \diag(\uplambda(\xx)) \adj{V}$ be a reduced rank-$r$ SVD decomposition, where $U \in \RR^{p_1 \times r}$ and $V \in \RR^{p_2 \times r}$ have orthonormal columns, and $\uplambda(\xx) \in (\RR_+ \setminus \ens{0})^{r}$ is the vector of singular values $(\uplambda_1(\xx),\cdots,\uplambda_r(\xx))$ in non-increasing order. From the partial smoothness of the nuclear norm at $\xx$ (Example~\ref{ex:nucM}) and its subdifferential, one can deduce that
\begin{align}
\label{eq:Tnuc}
&\Tgt_{\xx}(\Mm) = \T_{\xx} = \enscond{U A^* + B V^*}{ A \in \RR^{p_2 \times r}, B \in \RR^{p_1 \times r}  } \tandt \\
&\grad_{\Mm}\norm{\cdot}_*(\xx) = \e{\xx} = UV^* . \nonumber
\end{align}
It can be checked that the orthogonal projector on $\T_{\xx}$ is given by
\begin{align*}
\proj_{\T_{\xx}}W = U\adj{U} W + W V\adj{V} - U\adj{U} W V\adj{V} 
\end{align*}
Let $\xi \in \T_{\xx}$ and $W \in \S_{\xx}$. Then, from \citep[Section~4.5]{Absil13}, the Weingarten map reads
\begin{align}
\label{eq:weingnuc}
\Afk_{\xx}\pa{\xi,W} = W \adj{\xi} \adj{\xx^{+}} + \adj{\xx^{+}} \adj{\xi} W
\qwhereq
\adj{\xx^{+}} \eqdef U \diag(\uplambda(\xx))^{-1} \adj{V} .
\end{align}
In turn, from Fact~\ref{fact:gradhess}, the Riemannian Hessian of the nuclear norm reads
\begin{align*}
\hess_{\Mm}\norm{\cdot}_*(\xx)(\xi) &= \proj_{\T_{\xx}} \hess \widetilde{\norm{\cdot}_*}(\xx)(\proj_{\T_{\xx}} \xi) \\
&\quad + \proj_{\S_{\xx}} \grad\widetilde{\norm{\cdot}_*}(\xx) \adj{\xi} \adj{\xx^{+}} + \adj{\xx^{+}} \adj{\xi} \proj_{\S_{\xx}} \grad\widetilde{\norm{\cdot}_*}(\xx) ,
\end{align*}
where $\widetilde{\norm{\cdot}_*}$ is any smooth representative of the nuclear norm at $\xx$ on $\Mm$. Owing to the smooth transfer principle \citep[Corollary~2.3]{daniilidis2013orthogonal}, the nuclear norm has a $\Cdeux$-smooth (and even convex) representation on $\Mm$ near $\xx$ which is
\[
\widetilde{\norm{\xx'}_*} = \widetilde{\normu{\uplambda(\xx')}} = \sum_{i=1}^r \uplambda_i(\xx') .
\]
Combining this with \citep[Corollary~2.5]{Lewis95}, we then have $\grad\widetilde{\norm{\cdot}_*}(\xx) = UV^*$, and thus $\Afk_{\xx}\pa{\xi,\proj_{\S_{\xx}}\grad\widetilde{\norm{\cdot}_*}(\xx)} = 0$, or equivalently,
\begin{align}
\label{eq:nuchess}
\hess_{\Mm}\norm{\cdot}_*(\xx)(\xi) &= \proj_{\T_{\xx}} \hess \widetilde{\norm{\cdot}_*}(\xx)(\proj_{\T_{\xx}} \xi) .
\end{align}
The expression of the Hessian $\hess \widetilde{\norm{\cdot}_*}(\xx)$ can be obtained from the derivative of $UV^*$ using either \citep[Theorem~4.3]{CandesSVTSURE12} or \citep[Theorem~1]{DeledalleSVTSURE12} when $\xx$ is full-rank with distinct singular values, or from \citep[Theorem~3.3]{Lewis01} in the case where $\xx$ is symmetric with possibly repeated eigenvalues.  
\end{ex}

%%%%%
\begin{ex}[Indicator function of a partly smooth set $\Cc$]
\label{ex:sethess}
Let $\Cc$ be a closed convex and partly smooth set at $\xx \in \Cc$ relative to $\Mm$. From Example~\ref{ex:setM}, it is then clear that the zero-function is a smooth representative of $\iota_\Cc$ on $\Mm$ around $\xx$. In turn, the Riemannian gradient and Hessian of $\iota_\Cc$ vanish around $\xx$ on $\Mm$.
\end{ex}

%%%%%
\begin{ex}[$\J=\max(\norm{\cdot}-1,0)$]
\label{ex:maxnormhess}
Let $\xx \in \sph^{p-1}$. We have $\T_{\xx}=\pa{\RR\xx}^\perp$, and the orthogonal projector onto $\T_\xx$ is
\[
\proj_{\T_\xx} = \Id - \xx\transp{\xx} .
\]
The Weingarten map then reduces to
\[
\Afk_{\xx}\pa{\xi,v} = -\xi\dotp{\xx}{v} , \quad \xi \in \T_\xx \tandt v \in \S_\xx .
\]
Moreover, the zero-function is a smooth representative of $\J$ on $\sph^{p-1}$. It then follows that $\Q(\xx)=0$.
\end{ex}

%%% Local Variables:
%%% mode: latex
%%% TeX-master: "../AISM-DOF-Decomposable"
%%% End:

%% file: sections/local.tex
\section{Sensitivity Analysis of $\xxs(y)$}
\label{sec:local}

%%%%%%%%%%%%%%%%%%%%%%%%%%%%%%%%%%%%%%%%%%%%%%%%%%%%%%%%%%%%%%%%

In all the following, we consider the variational regularized problem~\eqref{eq-group-lasso}. We recall that $\J \in \lsc(\RR^p)$ and is partly smooth. We also suppose that the fidelity term fulfills the following conditions:
\eql{\label{hyp-f-reg}\tag{$C_{F}$}
	\foralls y \in  \RR^n, \quad   
	\F(\cdot,y) \in \Calt{2}(\RR^p)
	\qandq
	\foralls \xx \in  \RR^p, \quad \F(\xx,\cdot) \in \Calt{2}(\RR^n).
}
Combining \eqref{eq:covhess} and the first part of assumption~\eqref{hyp-f-reg}, we have for all $y \in \RR^n$
\eql{\label{eq:covhessF}
\Fxx(\xx,y)(\xx,y)\xi = \proj_{\T_\xx}\hess \F(\xx,y)\proj_{\T_\xx} + \Afk_{\xx}\pa{\xi,\proj_{\S_\xx} \grad F(\xx,y)}\proj_{\T_\xx} .
}
When $\Mm$ is affine or linear, equation \eqref{eq:covhessF} becomes
\eql{\label{eq:covhessFlin}
\Fxx(\xx,y)(\xx,y)\xi = \proj_{\T_\xx}\hess \F(\xx,y)\proj_{\T_\xx} .
}

\subsection{Restricted positive definiteness}
In this section, we aim at computing the derivative of the (set-valued) map $y \mapsto \xsoly(y)$ whenever this is possible. The following condition plays a pivotal role in this analysis.
\begin{defn}[Restricted Positive Definiteness] A vector $\xx \in \RR^p$ is said to satisfy the \emph{restricted positive definiteness condition} if, and only if,
\eql{\tag{$\CondInj{\xx}{y}$}\label{eq-injectivity-cond}
	\dotp{(\Fxx(\xx,y) + \Q(\xx))\xi}{\xi} > 0 \quad \forall ~ 0 \neq \xi \in \T_{\xx} .
}
\end{defn}

Condition~\eqref{eq-injectivity-cond} has a convenient re-writing in the following case.
\begin{lem}\label{lem:injectivity-cond}
Let $\J \in \lsc(\RR^p)$ be partly smooth at $\xx \in \RR^p$ relative to $\Mm$, and set $T = T_\xx$. Assume that $\Fxx(\xx,y)$ and $\Q(\xx)$ are positive semidefinite on $\T$. Then
\[
\text{\eqref{eq-injectivity-cond} holds if and only if }~~ \Ker( \Fxx(\xx,y) ) \cap \Ker(\Q(\xx)) \cap \T = \{ 0 \} .
\]
For instance, the positive semidefiniteness assumption is satisfied when $\Mm$ is an affine or linear subspace.
\end{lem}

When $\F$ takes the form \eqref{eq-fidelity-decompos} with $\F_0$ the squared loss, condition~\eqref{eq-injectivity-cond} can be interpreted as follows in the examples we discussed so far. 

\begin{ex}[Polyhedral penalty]
\label{ex:injpolyh}
Recall that a polyhedral penalty is partly smooth at $\xx$ relative to $\Mm=\xx+\T_{\xx}$. Combining this with Example~\ref{ex:lassohess}, condition \eqref{eq-injectivity-cond} specializes to
\[
\Ker(\XX_{\T_{\xx}}) = \ens{0} .
\]

\paragraph{Lasso}
Applying this to the Lasso (see~Example~\ref{ex:lassoM}), \eqref{eq-injectivity-cond} reads $\Ker(\XX_\Lambda)=\ens{0}$, with $\Lambda=\supp(\xx)$. This condition is already known in the literature, see for instance~\citep{2012-kachour-statsinica}.
\paragraph{General Lasso}
In this case, Example~\ref{ex:analassoM} entails that \eqref{eq-injectivity-cond} becomes
\[
\Ker(\XX) \cap \Ker(D_{\Lambda^c}^*)=\ens{0}, \qwhereq \Lambda = \supp(D^* \xx) .
\]
This condition was proposed in~\citep{vaiter-local-behavior}. 
%For the case where $D$ has orthonormal rows\footnote{These are a.k.a. Parseval tight frames, such as translation invariant wavelets or curvelets~\citep{mallat2009a-wav}.}, 
\end{ex}

\begin{ex}[Group Lasso]
\label{ex:injglasso}
For the case of the group Lasso, by virtue of Lemma~\ref{lem:unique}(ii) and Example~\ref{ex:glassohess}, one can see that condition~\eqref{eq-injectivity-cond} amounts to assuming that the system $\enscond{\XX_b \beta_b}{b \in \Bb, \xx_{b} \neq 0}$ is linearly independent. This condition appears in \citep{LiuZhang09} to establish $\ldeux$-consistency of the group Lasso. It goes without saying that condition~\eqref{eq-injectivity-cond} is much weaker than imposing that $\XX_\Lambda$ is full column rank, which is standard when analyzing the Lasso.
\end{ex}

\begin{ex}[General group Lasso]
\label{ex:injgglasso}
For the general group Lasso, let $I_{\xx}=\enscond{i}{b_i \in \Bb \tandt D^*_{b_i}\xx \neq 0}$, i.e.~the set indexing the active blocks of $D^*\xx$. Combining Example~\ref{ex:gglassoM} and Example~\ref{ex:gglassohess}, one has
\begin{align*}
\Ker(\Q(\xx)) &\cap \Ker(D_{\Lambda^c}^*) = \\ &\enscond{h \in \RR^p}{D_{b_i}^*h=0 ~ \forall i \notin I_{\xx} \tandt D^*_{b_i} h \in \RR ~ D^*_{b_i} \xx ~ \forall i \in I_{\xx}} ,
\end{align*}
where $\Lambda = \bs(D^* \xx)$. Indeed, $\delta_{D^*\xx}$ is a diagonal strictly positive linear operator, and $Q_{(D^*\xx)^\perp}$ is a block-wise linear orthogonal projector, and we get for $h \in \Ker(D_{\Lambda^c}^*)$,
\begin{align*}
h \in \Ker(\Q(\xx))	& \iff \dotp{h}{\Q(\xx)h} = 0 \\
			& \iff \dotp{D^*h}{\pa{\delta_{D^*\xx} \circ Q_{(D^*\xx)^\perp}} D^*h} = 0 \\
			& \iff \sum_{i \in I_{\xx}} \frac{\norm{\proj_{(D^*_{b_i}\xx)^\perp}(D^*_{b_i} h)}^2}{\norm{D^*_{b_i}\xx}} = 0 \\
			& \iff D^*_{b_i}\xx \in \RR ~ D^*_{b_i} \xx \quad \forall i \in I_{\xx} .
\end{align*}
In turn, by Lemma~\ref{lem:unique}(ii), condition~\eqref{eq-injectivity-cond} is equivalent to saying that $0$ is the only vector in the set
\[
\enscond{h \in \RR^p}{\XX h = 0 \tandt D^*_{b_i} h = 0 ~ \forall i \notin I_{\xx} \tandt D^*_{b_i} h \in \RR ~ D^*_{b_i} \xx ~ \forall i \in I_{\xx} } .
\]
Observe that when $D$ is a Parseval tight frame, i.e.~$DD^*=\Id$, the above condition is also equivalent to saying that the system $\enscond{\pa{\XX D}_{b_i} D^*_{b_i}\beta}{i \in I_{\xx}}$ is linearly independent.
\end{ex}

\begin{ex}[Nuclear norm]
\label{ex:injnuc}
We have seen in Example~\ref{ex:nuchess} that the nuclear norm has a $\Cdeux$-smooth representative which is also convex. It then follows from \eqref{eq:nuchess} that the Riemannian Hessian of the nuclear norm at $\xx$ is positive semidefinite on $\T_\xx$, where $\T_\xx$ is given in \eqref{eq:Tnuc}.

As far as $F$ is concerned, one cannot conclude in general on positive semidefiniteness of its Riemannian Hessian. Let's consider the case where $\xx \in \Sbf^p$, the vector space of real $p_1 \times p_1$ symmetric matrices endowed with the trace (Frobenius) inner product $\dotptr{\xx}{\xx'}=\tr(\xx\xx')$. From \eqref{eq:weingnuc} and \eqref{eq:covhessF}, we have for any $\xi \in \T_{\xx} \cap \Sbf^{p_1}$
\begin{align*}
\dotptr{\xi}{\Fxx(\xx,y)(\xi)} = &\dotptr{\xi}{\proj_{\T_{\xx}} \hess\F(\xx,y)(\proj_{\T_{\xx}} \xi)} \\
& + 2\dotptr{\xi U \diag(\uplambda(\xx))^{-1} \transp{U} \xi}{\proj_{\S_\xx}\Fx(\xx,y)} .
\end{align*}
Assume that $\xx$ is a global minimizer of \lasso, which by Lemma~\ref{lem:first-order}, implies that
\[
\proj_{\S_\xx}\Fx(\xx,y) = U_{\perp} \diag(\upgamma) \transp{U_{\perp}}
\]
where $U_{\perp} \in \RR^{n \times (p_1-r)}$ is a matrix whose columns are orthonormal to $U$, and $\upgamma \in [-1,1]^{p_1-r}$. We then get
\begin{align*}
\dotptr{\xi}{\Fxx(\xx,y)(\xi)} = &\dotptr{\xi}{\proj_{\T_{\xx}} \hess \F(\xx,y)(\proj_{\T_{\xx}} \xi)} \\
& + 2\dotptr{\transp{U_\perp}\xi U \diag(\uplambda(\xx))^{-1} \transp{U} \xi U_\perp}{\diag(\upgamma)} .
\end{align*}
It is then sufficient that $\xx$ is such that the entries of $\upgamma$ are positive for $\Fxx(\xx,y)$ to be indeed positive semidefinite on $\T$. In this case, Lemma~\ref{lem:injectivity-cond} applies.
\end{ex}

In a nutshell, Lemma~\ref{lem:injectivity-cond} does not always apply to the nuclear norm as $\Fxx(\xx,y)$ is not always guaranteed to be positive semidefinite in this case. One may then wonder whether there exist partly smooth functions $\J$, with a non-flat active submanifold, for which Lemma~\ref{lem:injectivity-cond} applies, at least at some minimizer of \lasso. The answer is affirmative for instance for the regularizer of Example~\ref{ex:maxnormM}.

%%%%%
\begin{ex}[$\J=\max(\norm{\cdot}-1,0)$]
\label{ex:injmaxnorm}
Let $\xx \in \sph^{p-1}$. From Example~\ref{ex:maxnormhess}, we have for $\xi \in \T_\xx$
\[
\dotp{\xi}{\Fxx(\xx,y)\xi} = \dotp{\xi}{\hess \F(\xx,y)\xi} - \norm{\xi}^2\dotp{\xx}{\Fx(\xx,y)} .
\]
Assume that $\xx$ is a global minimizer of \lasso, which by Lemma~\ref{lem:first-order}, implies that
\[
-\Fx(\xx,y) \in \xx [0,1] \Rightarrow -\dotp{\xx}{\Fx(\xx,y)} \in [0,1] .
\]
Thus, $\dotp{\xi}{\Fxx(\xx,y)\xi} \geq 0$, for all $\xi \in \T_\xx$. Since from Example~\ref{ex:maxnormhess}, $\Q(\xx)=0$, Lemma~\ref{lem:injectivity-cond} applies at $\xx$. Condition~\eqref{eq-injectivity-cond} then holds if, and only if, $\Fxx(\xx,y)$ is positive definite on $\T_\xx$. For the case of a quadratic loss, this is equivalent to 
\[
\ker(\XX) \cap \T_\xx = \ens{0} \quad or \quad \text{$\xx$ is not a minimizer of $\F(\cdot,y)$}.
\]
\end{ex}

\subsection{Sensitivity analysis: Main result}
Let us now turn to the sensitivity of any minimizer $\xsoly(y)$ of \lasso to perturbations of $y$. Because of non-smoothness of the regularizer $\J$, it is a well-known fact in sensitivity analysis that one cannot hope for a global claim, i.e. an everywhere smooth mapping\footnote{To be understood here as a set-valued mapping.} $y \mapsto \xsoly(y)$. Rather, the sensitivity behaviour is local. This is why the reason we need to introduce the following transition space $\Hh$, which basically captures points of non-smoothness of $\xsoly(y)$.

Let's denote the set of all possible partial smoothness active manifolds $\Mm_{\xx}$ associated to $\J$ as
\begin{equation}
\label{eq:Tt}
  \Mscr = \left\{ \Mm_\xx \right\}_{\xx \in \RR^p}.
\end{equation}
For any $\Mm \in \Mscr$, we denote $\Mc$ the set of vectors sharing the same partial smoothness manifold $\Mm$,
\begin{equation*}
  \Mc = \enscond{\xx' \in \RR^p}{\Mm_{\xx'}=\Mm}.
\end{equation*}
For instance, when $J=\norm{\cdot}_1$, $\Mc_\xx$ is the cone of all vectors sharing the same support as $\xx$.

\begin{defn}\label{defn:h}
  The \emph{transition space} $\Hh$ is defined as
  \begin{align*}
    \Hh = \bigcup_{\Mm \in \Mscr} \; \Hh_{\Mm},
    \qwhereq \Hh_{\Mm} &= \bd(\Pi_{n+p,n}(\Aa_{\Mm})) ,
  \end{align*}
  where $\Mscr$ is given by \eqref{eq:Tt}, and we denote  
  \eq{
    \Pi_{n+p,n} : 
    \left\{
      \begin{array}{ccc}
        \RR^n \times \Mc & \longrightarrow & \RR^n \\
        (y, \xx)  & \longmapsto & y  
      \end{array}
    \right.
  }
  the canonical projection on the first $n$ coordinates, $\bd \Cc$ is the boundary of the set $\Cc$, and 
  \begin{equation*}
    \Aa_{\Mm} =
    \enscond{
      (y, \xx) \in \RR^n \times \Mc
    }{
      - \Fx( \xx, y) \in \rbd \partial J(\xx)
    }.
  \end{equation*}
\end{defn}

\begin{rem}
\label{rem:transpace}
Before stating our result, some comments about this definition are in order. When $\bd$ is removed in the definition of $\Hh_{\Mm}$, we recover the classical setting of sensitivity analysis under partial smoothness, where $\Hh_{\Mm}$ contains the set of degenerate minimizers (those such that $0$ is in the relative boundary of the subdifferential of $\F(\cdot,y)+\J$). This is considered for instance in \citep{Bolte2011,LewisGenericConvex11} who studied sensitivity of the minimizers of $\xx \mapsto f_{\upnu}(\xx) \eqdef f(\xx)-\dotp{\upnu}{\xx}$ to perturbations of $\upnu$ when $f \in \lsc(\RR^{p})$ and partly smooth; see also \citep{LewisGenericSemialgebraic15} for the semialgebraic non-necessarily non-convex case. These authors showed that for $\upnu$ outside a set of Lebesgue measure zero, $f_\upnu$ has a non-degenerate minimizer with quadratic growth of $f_{\upnu}$, and for each $\bar{\upnu}$ near $\upnu$, the perturbed function $f_{\bar{\upnu}}$ has a unique minimizer that lies on the active manifold of $f_{\upnu}$ with quadratic growth of $f_{\bar{\upnu}}$. These results however do not apply to our setting in general. To see this, consider the case of \lasso where $\F$ takes the form \eqref{eq-fidelity-decompos} with $\F_0$ the quadratic (the same applies to other losses in the exponential family just as well). Then, \lasso is equivalent to minimizing $f_{\upnu}$, with $f=\J+\norm{\XX\cdot}^2$ and $\upnu=2\transp{\XX}y$. It goes without saying that, in general (i.e. for any $\XX$), a property valid for $\upnu$ outside a zero Lebesgue measure set does \textbf{not} imply it holds for $y$ outside a zero Lebesgue measure set. To circumvent such a difficulty, our key contribution is to consider the boundary of $\Hh_{\Mm}$. This turns out to be crucial to get a set of dimension potentially strictly less than $n$, hence negligible, as we will show under a mild o-minimality assumption (see Section~\ref{sec:dof}). However, doing so, uniqueness of the minimizer is not longer guaranteed.
\end{rem}

In the particular case of the Lasso (resp. general Lasso), i.e. $\F_0$ is the squared loss, $\J=\norm{\cdot}_1$ (resp. $\J=\norm{D^* \cdot}_1$), the transition space $\Hh$ specializes to the one introduced in~\citep{2012-kachour-statsinica} (resp. \citep{vaiter-local-behavior}). In these specific cases, since $\J$ is a polyhedral gauge, $\Hh$ is in fact a union of affine hyperplanes. The geometry of this set can be significantly more complex for other regularizers. For instance, for $\J = \norm{\cdot}_{1,2}$, it can be shown to be a semi-algebraic set (union of algebraic hyper-surfaces). Section~\ref{sec:dof} is devoted to a detailed analysis of this set $\Hh$.

We are now equipped to state our main sensitivity analysis result, whose proof is deferred to Section~\ref{sub:local}.

\begin{thm}\label{thm-local}
  Assume that~\eqref{hyp-f-reg} holds. Let $y \not\in \Hh$, and $\xxs(y)$ a solution of \lasso where $\J \in \lsc(\RR^p)$ is partly smooth at $\xxs(y)$ relative to $\Mm \eqdef \Mm_{\xxs(y)}$ and such that $(\CondInj{\xxs(y)}{y})$ holds.
  Then, there exists an open neighborhood $\neighb \subset \RR^n$ of $y$, and a mapping $\solm : \neighb \to \Mm$ such that
  \begin{enumerate}
  \item For all $\bar y \in \neighb$, $\solmB$ is a solution of $(\lassoB)$, and $\solm(y) = \xxs(y)$.
  \item the mapping $\solm$ is $\Calt{1}(\neighb)$ and
    \eql{\label{eq-differential}
      \foralls \bar y \in \neighb, \quad
      \jac \solm(\bar y) = -( \Fxx(\solm(\bar y),\bar y) + \Q(\solm(\bar y)) )^{+} 
       \proj_{\T_{\solm(\bar y)}} 
      \Fxy(\solm(\bar y),\bar y),
    }
    where $\Fxy(\xx,y)$ is the Jacobian of $\Fx(\xx,\cdot)$ with respect to the second variable evaluated at $y$.
  \end{enumerate}
\end{thm}
Theorem~\ref{thm-local}  can be extended to the case where the data fidelity is of the form $\F(\xx,\theta)$ for some parameter $\theta$, with no particular role of $y$ here.
%One now may wonder whether condition $(\CondInj{\xxs(y)}{y})$ is restrictive, and in particular, whether there exists always a solution $\xxs(y)$ such that it holds. In the following section, we give an affirmative answer in several cases of interest, with the proviso that the loss $\F_0$ is strictly convex. 

%% file: sections/well.tex
\section{Sensitivity Analysis of $\msol(y)$}
\label{sec:sensmu}

We assume in this section that $\F$ takes the form \eqref{eq-fidelity-decompos} with 
\eql{\label{eq-stric-cvx}\tag{$C_{\text{dp}}$}
	\foralls (\mu,y) \in \RR^n \times \RR^n, \quad
	\Fxxo(\mu,y) \text{ is positive definite.}
}
This in turn implies that $\F_0(\cdot,y)$ is strictly convex for any $y$ (the converse is obviously not true). Recall that this condition is mild and holds in many situations, in particular for some losses \eqref{eq:fidexp} in the exponential family, see Section~\ref{sec:block_regularizations} for details.

We have the following simple lemma.
\begin{lem}\label{lem:unique}
Suppose the condition~\eqref{eq-stric-cvx} is satisfied. The following holds,
\begin{enumerate}[label=(\roman*)]
\item All minimizers of~\eqref{eq-group-lasso} share the same image under $\XX$ and $\J$.
\item If the partial smoothness submanifold $\Mm$ at $\xx$ is affine or linear, then \eqref{eq-injectivity-cond} holds if, and only if, $\Ker( \XX ) \cap \Ker(\Q(\xx)) \cap \T = \{ 0 \}$, where $\T=\T_\xx$ and $\Q(\xx)$ is given in Fact~\ref{fact:gradhess}.
\end{enumerate}
\end{lem}

% Moreover, we assume that
% \eql{\label{eq-exist}\tag{$C_{\text{ex}}$}
%     \text{There always exists a solution $\xxs$ of \eqref{eq-group-lasso} such that $(\CondInj{\xxs}{y})$ holds.}
% }
% The following proposition shows that several regularizations satisfy~\eqref{eq-exist}.
% \begin{prop}\label{prop-exist}
%   Let $D \in \RR^{p \times q}$. Then $\normu{\cdot}$, $\normu{D^* \cdot}$, $\norm{\cdot}_{1,2}$ and $\norm{D^* \cdot}_{1,2}$ satisfy~\eqref{eq-exist}.
% \end{prop}
% In fact, this statement remains true for other regularizations such as polyhedral.

Owing to this lemma, we can now define the prediction 
\eql{\label{eq-predictor}
	\msol(y) = X \widehat{\xx}(y)
} 
without ambiguity given any solution $\widehat{\xx}(y)$, which in turn defines a single-valued mapping $\msol$. The following theorem provides a closed-form expression of the local variations of $\msol$ as a function of perturbations of $y$. For this, we define the following set that rules out the points $y$ where $(\CondInj{\xxs(y)}{y})$ does not hold for any any minimizer.

\begin{defn}[Non-injectivity set]
The \emph{Non-injectivity set} $\Gg$ is
\begin{align*}
  \Gg = \enscond{y \notin \Hh}{\text{$(\CondInj{\xxs(y)}{y})$ does not hold for any minimizer $\xxs(y)$ of~\lasso}} ~.
\end{align*}
\end{defn}

\begin{thm}\label{thm-div}
  Under assumptions~\eqref{hyp-f-reg} and \eqref{eq-stric-cvx},
  the mapping $y \mapsto \msol(y)$ is $\Calt{1}(\RR^n \setminus (\Hh \cup \Gg))$.
  Moreover, for all $y \not\in \Hh \cup \Gg$, 
  \eql{\label{eq:diverg}
    \diverg(\msol)(y) \eqdef \tr(\jac \msol(y)) = \tr(\De(y))
  }
  where
  \begin{align*}
    \De(y) = -\XX_\T ~ ( \Fxx(\msol(y),y) + \Q(\xxs(y)) )^{+} 
    ~ \transp{\XX_\T} ~ \Fxyo(\msol(y),y) , \\
    \Fxx(\msol(y),y) = \transp{\XX_\T} \Fxxo(\msol(y),y) \XX_\T + \Afk_{\xx}\pa{\cdot,\transp{\XX_\S}\Fxo(\msol(y),y)}
  \end{align*}
  and $\xxs(y)$ is any solution of \lasso such that $(\CondInj{\xxs(y)}{y})$ holds
  and $\J \in \lsc(\RR^p)$ is partly smooth at $\xxs(y)$ relative to $\Mm$, with $\T=\S^\perp=\T_{\xxs(y)}$.   
\end{thm}
This result is proved in Section~\ref{sub:dof}.

A natural question that arises is whether the set $\Gg$ is of full (Hausdorff) dimension or not, and in particular, whether there always exists a solution $\xxs(y)$ such that $(\CondInj{\xxs(y)}{y})$ holds, i.e. $\Gg$ is empty. Though we cannot provide an affirmative answer to this for any partly smooth regularizer, and this has to be checked on a case-by-case basis, it turns out that $\Gg$ is indeed empty for many regularizers of interest as established in the next result.
\begin{prop}\label{prop-exist}
The set $\Gg$ is empty when: 
\begin{enumerate}[label=(\roman*)]
\item $\J \in \lsc(\RR^p)$ is polyhedral, and in particular, when $\J$ is the Lasso, the general Lasso or the $\linf$ penalties.
\item $\J$ is the general group Lasso penalty, and a fortiori the group Lasso.
\end{enumerate}
\end{prop}
The proof of these results is constructive.\\

We now exemplify the divergence formula \eqref{eq:diverg} when $\F_0$ is the squared loss.
\begin{ex}[Polyhedral penalty]
\label{ex:polyhdf}
Thanks to Example~\ref{ex:lassohess}, it is immediate to see that \eqref{eq:diverg} boils down to
\begin{equation*}
  \diverg(\msol)(y) = \rank \XX_{\T_{\xxs(y)}} = \dim \T_{\xxs(y)}
\end{equation*}
where we used the rank-nullity theorem and that Lemma~\ref{lem:unique}(ii) holds at $\xxs(y)$, which always exists by Proposition~\ref{prop-exist}.
\end{ex}

\begin{ex}[Lasso and General Lasso]
Combining together Example~\ref{ex:analassoM} and Example~\ref{ex:polyhdf} yields
\begin{equation*}
  \diverg(\msol)(y) = \dim \Ker(D^*_{\Lambda^c}), \quad \Lambda = \supp(D^* \xxs(y)) ~,
\end{equation*}
where $\xxs(y)$ is such that Lemma~\ref{lem:unique}(ii) holds. For the Lasso, Example~\ref{ex:lassoM} allows to specialize the formula to
\begin{equation*}
  \diverg(\msol)(y) = \abs{\supp(\xxs(y))}.
\end{equation*}
The general Lasso case was investigated in~\citep{vaiter-local-behavior} and \citep{tibshirani2012dof}, and the Lasso in \citep{2012-kachour-statsinica} and  \citep{tibshirani2012dof}.
\end{ex}

\begin{ex}[$\linf$ Anti-sparsity]
By virtue of Example~\ref{ex:polyhdf} and Example~\ref{ex:linfM}, we obtain in this case
\begin{equation*}
  \diverg(\msol)(y) = p - \abs{I} + 1, \qwhereq
	I = \enscond{i}{\xxs_i(y) = \normi{\xxs(y)}}
\end{equation*}
and $\xxs(y)$ is such that Lemma~\ref{lem:unique}(ii) holds, and such a vector always exists by Proposition~\ref{prop-exist}.
\end{ex}

\begin{ex}[Group Lasso and General Group Lasso]
\label{ex:divglasso}
For the general group Lasso, piecing together Example~\ref{ex:gglassoM} and Example~\ref{ex:gglassohess}, it follows that
\begin{equation*}
  \diverg(\msol)(y) = \tr\pa{\XX_{\T}\pa{\transp{\XX_{\T}}\XX_{\T} + \proj_{\T} D\pa{\delta_{D^* \xxs(y)} \circ Q_{(D^* \xxs(y))^\perp}}D^*\proj_{\T}}^{+}\transp{\XX_{\T}}}
\end{equation*}
where $\T=\Ker(D^*_{\Lambda^c})$, $\Lambda = \bs(D^* \xxs(y))$, and $\xxs(y)$ is such that Lemma~\ref{lem:unique}(ii) holds; such a vector always exists by Proposition~\ref{prop-exist}. For the group Lasso, we get using Example~\ref{ex:glassoM} that
\begin{equation*}
  \diverg(\msol)(y) = \tr\pa{\XX_\Lambda\pa{\transp{\XX_\Lambda}\XX_\Lambda + \bpa{\delta_{D^* \xxs(y)} \circ Q_{(D^* \xxs(y))^\perp}}_{\Lambda,\Lambda}}^{-1}\transp{\XX_\Lambda}}
\end{equation*}
where $\bpa{\delta_{D^* \xxs(y)} \circ Q_{(D^* \xxs(y))^\perp}}_{\Lambda,\Lambda}$ is the submatrix whose rows and columns are those of $\delta_{D^* \xxs(y)} \circ Q_{(D^* \xxs(y))^\perp}$ indexed by $\Lambda=\bs(\xxs(y))$. This result was proved in~\citep{vaiter-icml-workshops} in the overdetermined case. An immediate consequence of this formula is obtained when $\XX$ is orthonormal\footnote{Obviously, Lemma~\ref{lem:unique}(ii) holds in such a case at the unique minimizer $\xxs(y)$.}, in which case one recovers the expression of \citet{yuan2006model},
\begin{equation*}
  \diverg(\msol)(y) = \abs{\Lambda} - \sum_{b \in \Bb, D^*_b \xxs(y) \neq 0} \frac{\abs{b}-1}{\norm{y_b}} ~.
\end{equation*}
The general group Lasso formula is new to the best of our knowledge, and will be illustrated in the numerical experiments on the isotropic 2-D total variation regularization widely used in image processing.
\end{ex}

We could also provide a divergence formula for the nuclear norm, but as we discussed in Example~\ref{ex:injnuc}, we cannot always guarantee the existence of a solution that satisfies $(\CondInj{\xxs(y)}{y})$. However, one can still find other partly smooth functions $\J$ with a non-flat submanifold for which this existence can be certified. The function of Example~\ref{ex:maxnormM} is again a prototypical example.
\begin{ex}[$\J=\max(\norm{\cdot}-1,0)$]
\label{ex:divmaxnorm} 
For $\xx \in \sph^{p-1}$. If $\xx$ is a minimizer of \lasso is not a minimizer of $F(\cdot,y)$, from Example~\ref{ex:injmaxnorm}, we have that $\Fxx(\xx,y)$  is positive definite on $\T=\T_\xx$. Thus, we get for the case of the squared loss, that
\[
\diverg(\msol)(y) = \tr\pa{\XX_{\T}\pa{\transp{\XX_{\T}}\XX_{\T} + \proj_{\T} \dotp{\XX\xx}{y-\XX\xx}}^{+}\transp{\XX_{\T}}} .
\]
\end{ex} 

%%% Local Variables:
%%% mode: latex
%%% TeX-master: "../AISM-DOF-Decomposable"
%%% End:

%% file: sections/gsure.tex
\section{Degrees of Freedom and Unbiased Risk Estimation}
\label{sec:dof}

From now on, we will assume that
\begin{align}
\label{hyp-tt} \tag{$C_\Mscr$}
	\text{the set $\Mscr$ is finite.}
\end{align}

Assumption~\eqref{hyp-tt} holds in many important cases, including the examples discussed in the paper: polyhedral penalties (e.g. the Lasso, general Lasso or $\linf$-norm), as well as for the group Lasso and its general form. 
%However, it precludes the case of the nuclear norm (also known as the trace norm). Our forthcoming results thus do not cover the latter.

Throughout this section, we use the same symbols to denote weak derivatives (whenever they exist) as for derivatives. Rigorously speaking, the identities have to be understood to hold Lebesgue-a.e.~\citep{EvansGariepy92}.

So far, we have shown that outside $\Hh \cup \Gg$, the mapping $y \mapsto \msol(y)$ enjoys (locally) nice smoothness properties, which in turn gives closed-form formula of its divergence. To establish that such formula holds Lebesgue a.e., a key argument that we need to show is that $\Hh$ is of negligible Lebesgue measure. This is where o-minimal geometry enters the picture. In turn, for $Y$ drawn from some appropriate probability measures with density with respect to the Lebesgue measure, this will allow us to establish unbiasedness of quadratic risk estimators. 

%%%%%%%%%%%%%%%%%%%%%%%%%%%%%%%%%%%%
\subsection{O-minimal Geometry}

Roughly speaking, to be able to control the size of $\Hh$, the function $\J$ cannot be too oscillating in order to prevent pathological behaviours. We now briefly recall here the definition. Some important properties of o-minimal structures that are relevant to our context together with their proofs are collected in Section~\ref{sec-omin}. The interested reader may refer to~\citep{van-den-Dries-omin-book,coste1999omin} for a comprehensive account and further details on o-minimal structures.

\begin{defn}[Structure]\label{defn-omin}
A \emph{structure} $\omin$ expanding $\RR$ is a sequence $(\omin_k)_{k \in \NN}$ which satisfies the
following axioms:
	\begin{enumerate}
		\item Each $\omin_k$ is a Boolean algebra of subsets of $\RR^k$, with $\RR^k \in \omin_{k}$.
		\item Every semi-algebraic subset of $\RR^k$ is in $\omin_k$.
		\item If $A \in \omin_k$ and $B \in \omin_{k'}$, then $A \times B \in \omin_{k+k'}$.
		\item If $A \in \omin_{k+1}$, then $\Pi_{k+1,k}(A) \in \omin_k$, where $\Pi_{k+1,k}: \RR^{k+1} \to \RR^k$ is the projection on the first $k$ components.
	\end{enumerate}	
	The structure $\omin$ is said to be \emph{o-minimal} if, moreover, it satisfies
	\begin{enumerate}
		\item[5.] (o-minimality) Sets in $\omin_1$ are precisely the finite unions of intervals and points of~$\RR$.
	\end{enumerate}
\end{defn}
In the following, a set $A \in  \omin_k$ is said to be definable.

\begin{defn}[Definable set and function]
Let $\omin$ be an o-minimal structure. The elements of $\omin_k$ are called the \emph{definable subsets} of $\RR^p$, i.e. $\Om \subset \RR^k$ is definable if $\Om \in \omin_k$. A map $f : \Om \rightarrow \RR^m$ is said to be definable if its graph $\Gg(f) = \enscond{(x,u) \in \Om \times \RR^m}{u=f(x)} \subseteq \RR^{k} \times \RR^{m}$ is a definable subset of $\RR^{k} \times \RR^m$ (in which case $m$ times applications of axiom 4 implies that $\Omega$ is definable).
\end{defn}

A fundamental class of o-minimal structures is the collection of semi-algebraic sets, in which case axiom~4 is actually a property known as the Tarski-Seidenberg theorem~\citep{coste2002intro}. For example, in the special case where $q$ is a rational number, $J=\norm{\cdot}_q$ is semi-algebraic. When $q \in \RR$ is not rational, $\norm{\cdot}_q$ is not semi-algebraic, however, it can be shown to be definable in an o-minimal structure. To see this, we recall from \citep[Example~5 and Property~5.2]{vandenDriesMiller96} that there exists a (polynomially bounded) o-minimal structure that contains the family of functions $\enscond{t > 0}{t^\gamma, \gamma \in \RR}$ and restricted analytic functions. Functions $\F_0$ that correspond to the exponential family losses introduced in Example~\ref{exp-glm} are also definable. \\

Our o-minimality assumptions requires the existence of an o-minimal structure $\omin$ such that
\begin{equation}
\label{eq-condition-omin}\tag{$C_{\omin}$}
\begin{split}
	\text{$F$, $\J$ and $\Mm$, $\forall \Mm \in \Mscr$, are definable in } \omin.
\end{split}
\end{equation}

%%%%%%%%%%%%%%%%%%%%%%%%%%%%%%%%%%%%
\subsection{Degrees of Freedom and Unbiased Risk Estimation}
We assume in this section that $\F$ takes the form \eqref{eq-fidelity-decompos} and that 
\eql{\label{eq-strong-convex}\tag{$C_{\mathrm{sconv}}$}
	\foralls y \in \RR^n, \quad
	\F_0(\cdot,y) \text{ is strongly convex with modulus } \tau
}
and
\begin{equation}\label{cnd-hess-bord}\tag{$C_{L}$}
      \exists L>0, \quad
      \sup_{(\mu,y) \in \RR^n\times\RR^n}\norm{\Fxyo(\mu,y)} \leq L .
\end{equation}

Obviously, assumption \eqref{eq-strong-convex} implies \eqref{eq-stric-cvx}, and thus the claims of the previous section remain true. Moreover, this assumption holds for the squared loss, but also for some losses of the exponential family \eqref{eq:fidexp}, possibly adding a small quadratic term in $\beta$. As far as assumption \eqref{cnd-hess-bord} is concerned, it is easy to check that it is fulfilled with $L=1$ for any loss of the exponential family~\eqref{eq:fidexp}, since $\Fxyo(\mu,y) = -\Id$.

\myparagraph{Non-linear Gaussian regression.}
Assume that the observation model \eqref{eq:linear-problem} specializes to $Y \sim \Nn(h(\XX\xx_0),\si^2\Id_n)$, where $h$ is Lipschitz continuous.

\begin{thm}\label{thm-dof}
	The following holds.
	\begin{enumerate}[label=(\roman*)] 
	\item Under condition \eqref{eq-condition-omin}, $\Hh$ is of Lebesgue measure zero;
	\item Under conditions \eqref{eq-strong-convex} and \eqref{cnd-hess-bord}, $h \circ \msol$ is Lipschitz continuous, hence weakly differentiable, with an essentially bounded gradient.
	\item If conditions~\eqref{eq-condition-omin}, \eqref{eq-strong-convex}, \eqref{hyp-f-reg} and \eqref{cnd-hess-bord} hold, and~$\Gg$ is of zero-Lebesgue measure, then, 	
	\begin{enumerate}[label=(\alph*)]
	\item $\widehat \DOF = \tr(\jac h(\msol(Y))\De(Y))$ is an unbiased estimate of $\DOF=\EE(\diverg(h \circ \msol(Y)))$, where $\De(Y)$ is as given in Theorem~\ref{thm-div}.
	\item The $\SURE$
	\begin{align}
	\label{eq:sure}
    \SURE(h \circ \msol)(Y) = &
    \norm{Y - h(\msol(Y))}^2
    + 2 \sigma^2 \widehat \DOF
    - n \sigma^2
	\end{align}
	is an unbiased estimator of the risk $\EE\pa{\norm{h(\msol(Y)) - h(\mu_0)}^2}$.
	\end{enumerate}
	\end{enumerate}
\end{thm}
This theorem is proved in Section~\ref{sub:sure}.

\myparagraph{GLM with the continuous exponential family.}
Assume that the observation model \eqref{eq:linear-problem} corresponds to the GLM with a distribution which belongs to a continuous standard exponential family as parameterized in \eqref{eq:pdfexp}. From the latter, we have 
\eq{
	\grad \log B(y)=\pa{\frac{\partial \log B_i(y_i)}{\partial y_i}}_i.
}

\begin{thm}\label{thm-dof-exp}
	Suppose that conditions~\eqref{eq-condition-omin}, \eqref{eq-strong-convex}, \eqref{hyp-f-reg} and \eqref{cnd-hess-bord} hold, and~$\Gg$ is of zero-Lebesgue measure.
	Then,
	\begin{enumerate}[label=(\roman*)] 
	\item $\widehat \DOF = \tr(\De(Y))$ is an unbiased estimate of $\DOF=\EE(\diverg(\msol(Y)))$.
	\item The $\SURE$
	\begin{align}
	\label{eq:gsureexp}
    \SURE(\msol)(Y) = &
    \norm{\grad \log B(Y) - \msol(Y)}^2
    + 2 \widehat \DOF
    - (\norm{\grad \log B(Y)}^2 - \norm{\mu_0}^2)
	\end{align}
	is an unbiased estimator of the risk $\EE\pa{\norm{\msol(Y) - \mu_0}^2}$.
	\end{enumerate}
\end{thm}
This theorem is proved in Section~\ref{sub:sure}. Recall from Section~\ref{sec:sensmu} that there are many regularizers where $\Gg$ is indeed empty, and for which Theorem~\ref{thm-dof} and \ref{thm-dof-exp} then apply.\\

Though $\SURE(\msol)(Y)$ depends on $\mu_0$, which is obviously unknown, it is only through an additive constant, which makes it suitable for parameter selection by risk minimization. Moreover, even if it is not stated here explicitly, Theorem~\ref{thm-dof-exp} can be extended to unbiasedly estimate other measures of the risk, including the {\it projection} risk, or the {\it estimation} risk (in the full rank case) through the Generalized Stein Unbiased Risk Estimator as proposed in \citep[Section~IV]{eldar-gsure}, see also \citep{vaiter-local-behavior} in the Gaussian case.

%%% Local Variables: 
%%% TeX-master: "../l1l2-variations-dof.tex"
%%% End: 

%% file: sections/experiments.tex
\section{Simulation results}
\label{sec:sim}

%\todo{Isotropic 2D TV experiments (CS and deconvolution). Following reviewer, add a figure showing the evolution of the dof as a function of $\lambda$ for the group Lasso and its general version (e.g. isotropic TV). }
\paragraph{Experimental setting.}
In this section, we illustrate the efficiency of the proposed DOF estimator on
a parameter selection problem in the context of some imaging inverse problems.
More precisely, we consider the linear Gaussian regression model
$
  Y \sim \Nn(\XX\xx_0, \sigma^2 \Id_n)
$
where $\xx_0 \in \RR^{p = p_1 \times p_2}$ is a column-vectorized
version of an image defined on a 2-D discrete grid of size $p_1 \times p_2$.
The estimation of $\xx_0$ is achieved by solving~\eqref{eq-group-lasso} with
\eq{
  F(\xx, y) = F_0(\XX\xx, y) = \norm{\XX \xx - y}^2
  \qandq
  J(\xx) = \lambda \norm{D^* \xx}_{1,2}
}
where $D^* \xx \in \RR^{p \times 2}$ is the 2-D discrete gradient vector field of the image $\xx$, and $\lambda > 0$ is the regularization parameter. Clearly, $J$ is the isotropic total variation regularization \citep{rudin1992nonlinear}, which is a special case of the general group Lasso penalty~\eqref{lun-deux-analysis} for blocks of size $2$. 
 
We aim at proposing an automatic and objective way to choose $\lambda$. This can be achieved typically by minimizing the SURE given in~\eqref{eq:sure} with $h$ being the identity, i.e.% the following form
\eq{
\SURE(\msol)(Y) =
    \norm{Y - \msol(Y)}^2 + 2 \sigma^2 \widehat \DOF - n \sigma^2
}
where $\widehat \DOF = \tr(\De(Y))$ according to Theorem~\ref{thm-dof}(iii)-(a), and the expression of $\De(Y)$ is obtained from that of the general group Lasso in Example~\ref{ex:divglasso} with $D^*$ the discrete 2-D gradient operator, and $-D$ is the discrete 2-D divergence operator. Owing to Proposition~\ref{prop-exist}(ii) and Theorem~\ref{thm-dof}(iii), the given SURE is indeed an unbiased estimator of the prediction risk.

%According to Theorem \ref{thm-div}, for any observation $y$ of $Y$,
%the matrix $\Delta(y)$ reads
%\begin{align}
%  & \Delta(y)
%  =
%  \XX_T(\transp{\XX_T \XX_T} - \lambda \diverg_I \delta_{\nabla_I \xx} P_{(\nabla_I \xx)^\top} \nabla_I)^{+} \transp{\XX_T} \\
%\hspace{-1cm}\qwithq &
%  T = \Ker \nabla_{I^c}
%  \qandq
%  I = \enscond{i}{\exists b \in \mathcal{B}, (\nabla \xx)_b \ne 0 \qandq i \in b}
%\end{align}
%where $\diverg = -\nabla^*$ is the two dimensional discrete divergence operator.
As the image size $p$ can be large, the exact computation of $\tr(\De(y))$
can become computationally intractable. Instead, we devise an approach based on
Monte-Carlo (MC) simulations \citep[see,][for more details]{vonesch2008sure}, that is
\eq{
  \widehat\DOF^{\mathrm{MC}}(z) = \dotp{z}{\Delta(Y) z}
}
with $z$ a realization of $Z \sim \Nn(0, \Id_n)$. It is clear that $\EE_Z\pa{\widehat\DOF^{\mathrm{MC}}(Z)} = \widehat\DOF$.

It remains to compute the vector $\Delta(y) z$. This is achieved by taking $\Delta(y) z = \XX \nu$, where $\nu$ is a solution of
\begin{align*}
  \pa{\transp{\XX} \XX + \lambda D \bpa{\delta_{D^* \xxs(y)} \circ Q_{(D^* \xxs(y))^\perp}} D^*} \nu = \transp{\XX} z
  \qsubjq \nu \in T ,
\end{align*}
where we recall that $\T=\Ker(D^*_{\Lambda^c})$, $\Lambda = \bs(D^* \xxs(y))$.
Taking into account the constraint on $T$ through its Lagrange multiplier $\zeta$, solving for $\nu$ boils down to solving
the following linear system with a symmetric and positive-definite matrix
\begin{align}\label{eq-sdp-dof}
  \begin{pmatrix}
    \transp{\XX} \XX + \lambda D \bpa{\delta_{D^* \xxs(y)} \circ Q_{(D^* \xxs(y))^\perp}} D^*
    & ~~ &
    D_{\Lambda^c}\\
    D^*_{\Lambda^c} & ~~ & 0\\
  \end{pmatrix}
  \begin{pmatrix}
    \nu \\
    \zeta\\
  \end{pmatrix}
  =
  \begin{pmatrix}
    \transp{\XX} z\\
    0\\
  \end{pmatrix}.
\end{align}

\begin{figure}[!t]
\centering
\subfigure[]{\includegraphics[width=0.32\linewidth,viewport=1 1 41 33,clip]{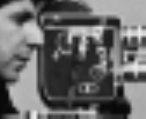}}\hfill%
\subfigure[]{\includegraphics[width=0.32\linewidth,viewport=1 1 41 33,clip]{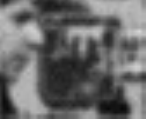}}\hfill%
\subfigure[]{\includegraphics[width=0.32\linewidth,viewport=1 1 41 33,clip]{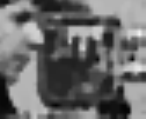}}\\
\subfigure[]{\includegraphics[width=0.32\linewidth]{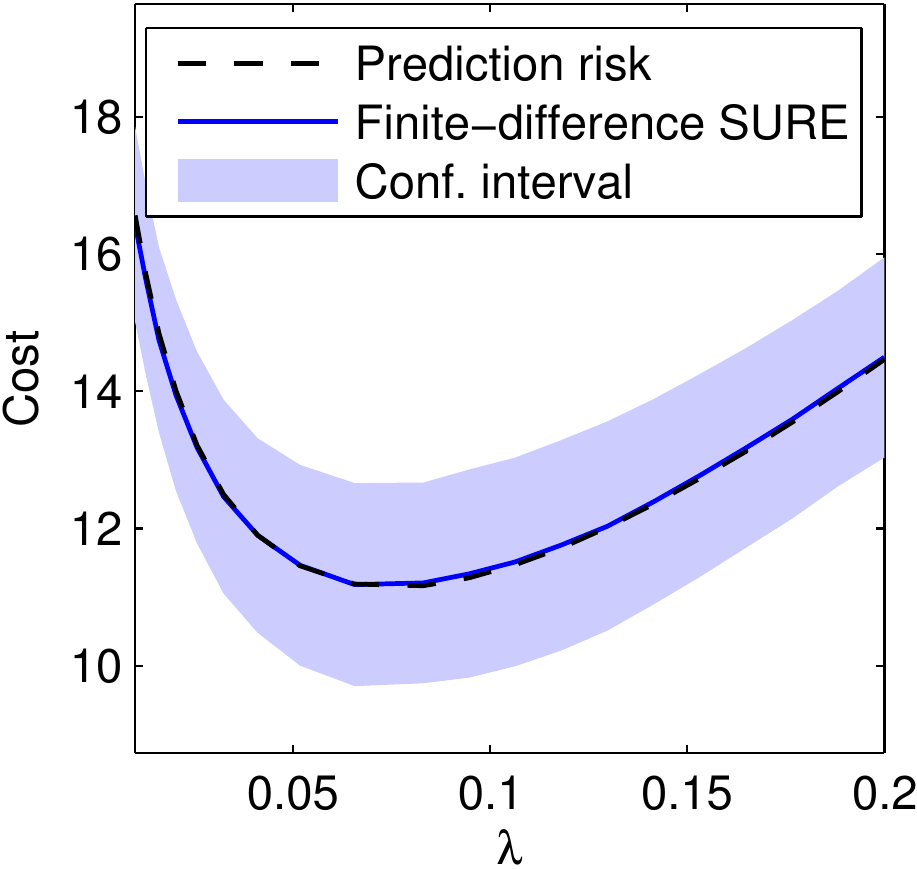}}\hfill%
\subfigure[]{\includegraphics[width=0.32\linewidth]{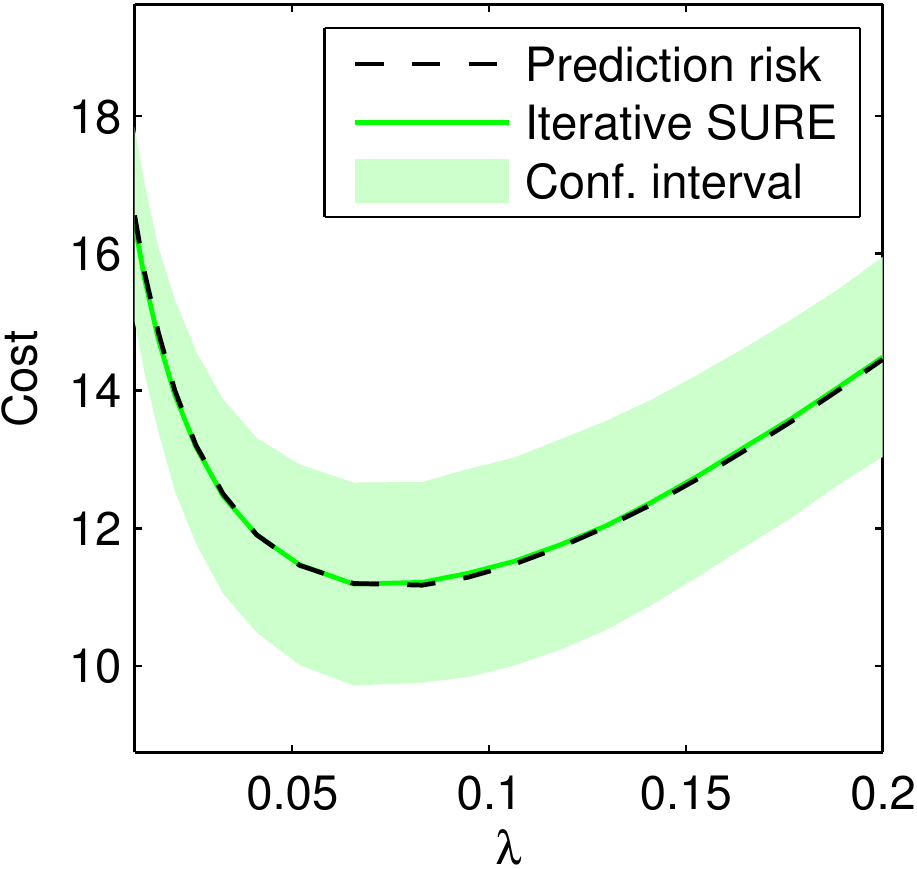}}\hfill%
\subfigure[]{\includegraphics[width=0.32\linewidth]{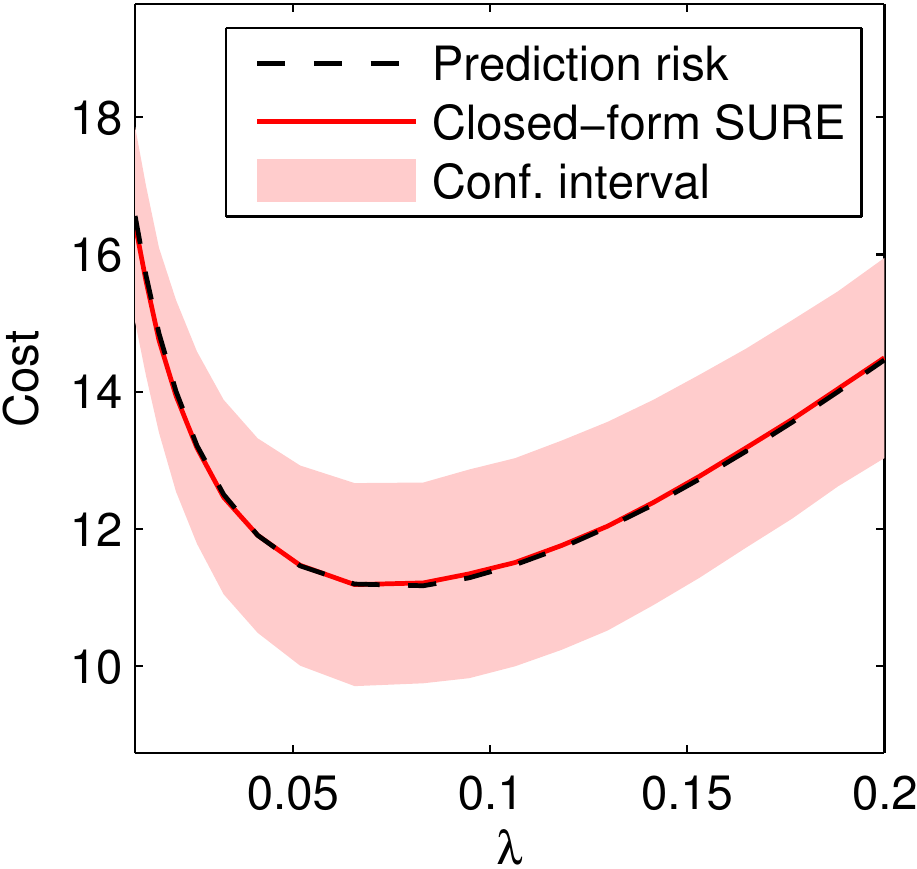}}
\caption{
  (a) Original image $\xx_0$.
  (b) Blurry observation $y$.
  (c) $\xxs(y)$ obtained for the value of $\lambda$ minimizing the SURE estimate.
  (d-f) Prediction risk, average SURE and its confidence interval ($\pm$ standard deviation)
  as a function of $\lambda$ respectively for
  the finite difference approach \citep{ramani2008montecarlosure},
  the iterative approach \citep{vonesch2008sure},
  and our proposed approach.
}
\label{fig:sure-deconvolution}
%\vspace{1em}
\end{figure}

\begin{figure}[!t]
\centering
\subfigure[]{\includegraphics[width=0.32\linewidth,viewport=1 1 41 33,clip]{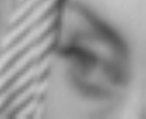}}\hfill%
\subfigure[]{\includegraphics[width=0.32\linewidth,viewport=1 1 41 33,clip]{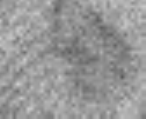}}\hfill%
\subfigure[]{\includegraphics[width=0.32\linewidth,viewport=1 1 41 33,clip]{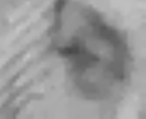}}\\
\subfigure[]{\includegraphics[width=0.32\linewidth]{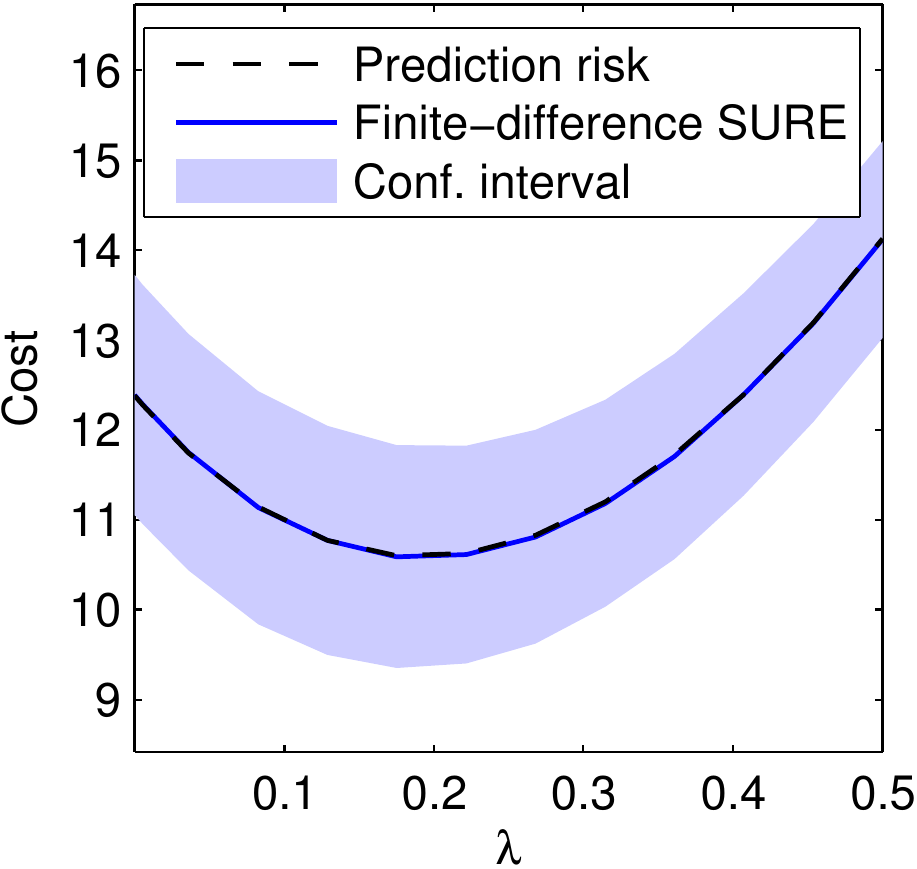}}\hfill%
\subfigure[]{\includegraphics[width=0.32\linewidth]{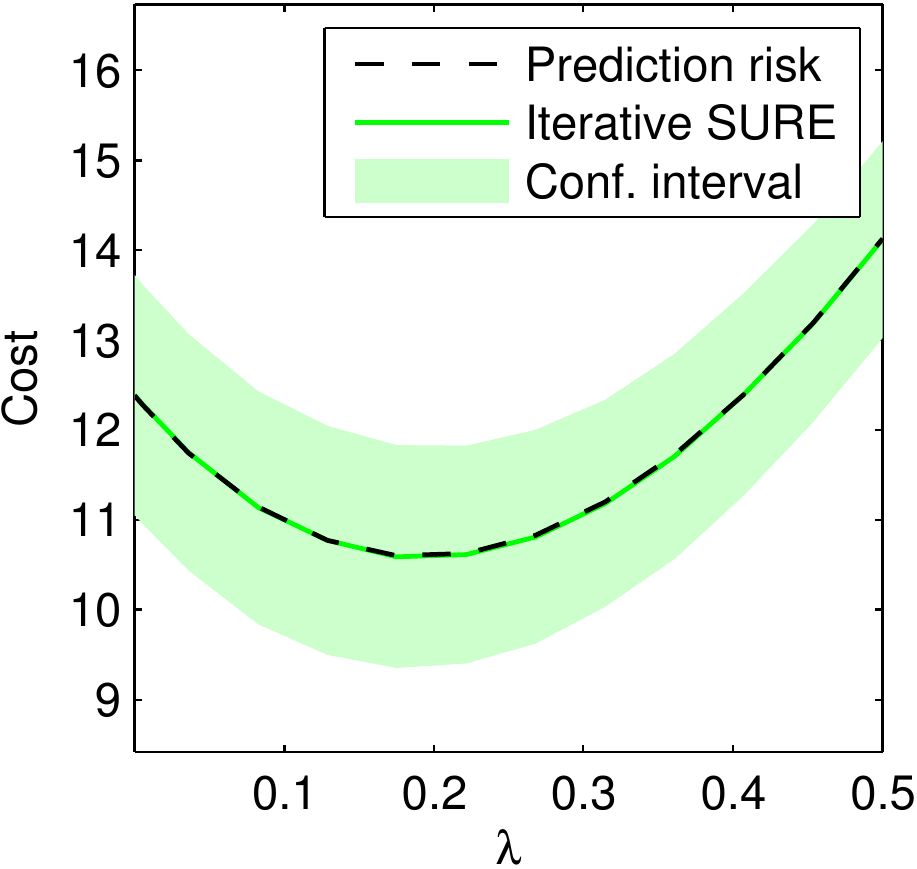}}\hfill%
\subfigure[]{\includegraphics[width=0.32\linewidth]{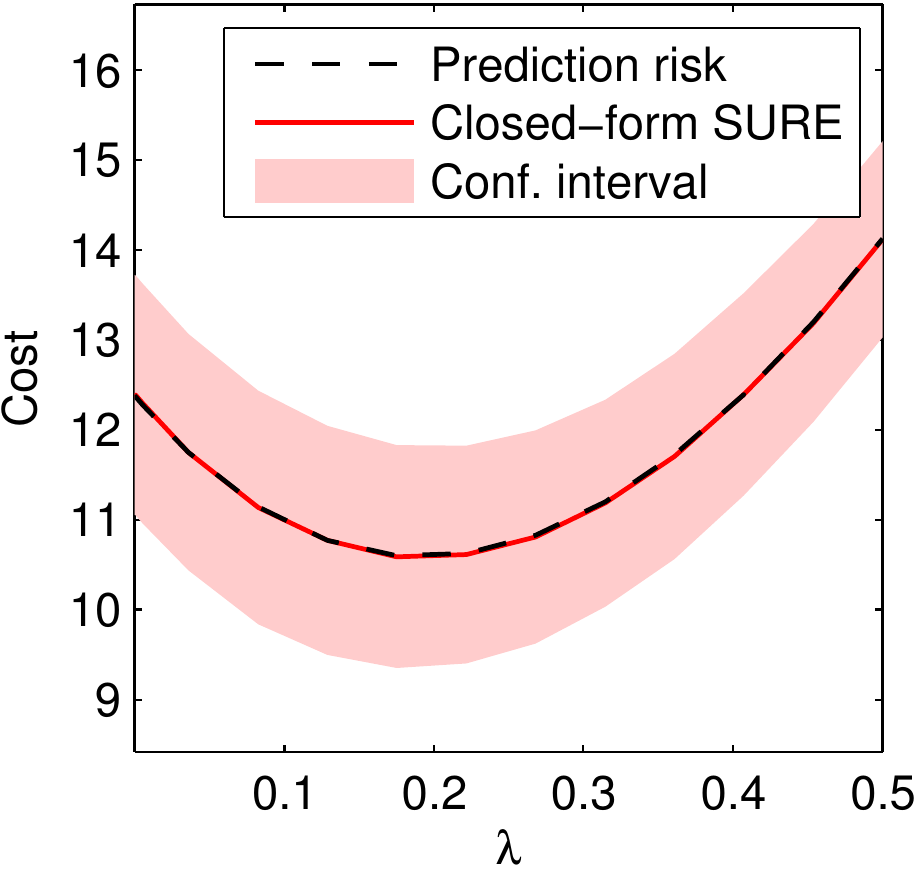}}
\caption{
  (a) Original image $\xx_0$.
  (b) Least squares estimate $\XX^+ y$.
  (c) $\xxs(y)$ obtained for the value of $\lambda$ minimizing the SURE estimate.
  (d-f) Prediction risk, average SURE and its confidence interval ($\pm$ standard deviation)
  as a function of $\lambda$ respectively for
  the finite difference approach \citep{ramani2008montecarlosure},
  the iterative approach \citep{deledalle2014stein},
  and our proposed approach.
}
\label{fig:sure-cs}
%\vspace{1em}
\end{figure}

\paragraph{Numerical solvers.}
In all experiments, optimization problem \eqref{eq-group-lasso} was solved
using Douglas-Rachford proximal splitting algorithm \citep{combettes2007douglas} with $2\cdot10^4$ iterations.
%The co-support $I$ is thus identified from the support of the auxiliary vector $\gamma$.
Once the support $\Lambda$ is identified with sufficiently high accuracy, the linear problem \eqref{eq-sdp-dof} is solved
using the generalized minimal residual method \citep[GMRES,][]{saad1986gmres}
%with a restart each $400$ iterations
with a relative accuracy of $10^{-7}$.

Our proposed SURE estimator is compared
for different values of $\lambda$
with the approach of \citep{ramani2008montecarlosure} based on finite difference
approximations, as well as
the approaches of \citep{vonesch2008sure,deledalle2014stein} based
on iterative chain rule differentiations.
All curves are averaged on $40$ independent realizations of $Y$ and $Z$ and their
corresponding confidence intervals at $\pm$ their standard deviation are displayed.

\paragraph{Deconvolution.}
We first consider an image of size $p = 34 \times 42$
with grayscale values ranging in $[0, 255]$
obtained from a close up of the standard {\it cameraman} image.
$\XX$ is a circulant matrix 
representing a periodic discrete convolution with a Gaussian kernel of width
$1.5$ pixel. The observation $y$ is finally obtained by adding a zero-mean white Gaussian noise with $\sigma = 5$.
Figure~\ref{fig:sure-deconvolution} depicts the evolution of the prediction risk and its SURE estimates as a function of $\lambda$.

\paragraph{Compressive sensing.}
We next consider an image of size $p = 34 \times 42$
with grayscale values ranging in $[0, 255]$
obtained from a close up of the standard {\it barbara} image. 
Now, $\XX$ is a matrix corresponding to the composition of a periodic discrete convolution with a square kernel, and a random sub-sampling matrix with $n/p = 0.5$. The noise standard deviation is again $\sigma = 5$.
Figure~\ref{fig:sure-cs} shows the evolution of the prediction risk and its SURE estimates as a function of $\lambda$.

\paragraph{Discussion.}
The three approaches seem to provide the same results with
average SURE curves that align very tightly with those of the prediction risk, with
relatively small standard deviation compared to the range of variation of the prediction risk.
%Note that for larger images, {\it e.g.}, $p = 256 \times 256$, or
%with much more realizations of $Y$ and $Z$,
%the average curves of the SURE would be perfectly aligned on the prediction risk and
%their confidence intervals would become much narrower.

It is worth observing that the SURE obtained with finite differences \citep{ramani2008montecarlosure} 
or with iterative differentiations \citep{vonesch2008sure,deledalle2014stein} estimate the risk at the last iterate
provided by the optimization algorithm to solve~\eqref{eq-group-lasso}, which is not exactly  $\xxs(y)$ in general.
In fact, what is important is not $\xxs(y)$ by itself but rather its group support $\Lambda$. 
Thus, provided $\Lambda$ has been perfectly identified, the three approaches
provide, as observed, the same estimate of the risk up to machine precision.
It may then be important to run the solver with a large number of iterations
in order to provide an accurate estimation of the risk.
Even more important, solutions of \eqref{eq-sdp-dof} should be
accurate enough to avoid bias in the estimation.
The choice of $2 \cdot 10^4$ iterations for Douglas-Rachford and
relative accuracy of $10^{-7}$ for GMRES appears in our simulations as a
good trade-off between negligible bias and reasonable computational time.
%For problems involving larger images (larger $p$),
%the number of iterations should probably be increased and
%the resolution of \eqref{eq-sdp-dof} with GMRES
%would become prohibitive for such a precision.
%It is then worth mentioning that our proposed expression of the DOF is
%of main interest for its theoretical analysis, but
%for computational efficiency,
%the two other approaches would be preferred as they can deal
%with much larger images.

%% file: sections/proofs.tex
\section{Proofs} % (fold)
\label{sec:proofs}

This section details the proofs of our results.

%%%%%%%%%%%%%%%%%%%%%%%%%%%%%%%%%%%%%%%%%%%%%%%
\subsection{Preparatory lemma}

By standard arguments of convex analysis, the following lemma gives the first-order sufficient and necessary optimality condition of a minimizer of \eqref{eq-group-lasso}.

\begin{lem}\label{lem:first-order}
  A vector $\xxs(y) \in \RR^p$  is a minimizer of \eqref{eq-group-lasso} if, and only if, 
  \eq{
  	- \Fx(\xxs(y),y) \in \partial \J(\xxs(y)).
  }
  If $\J$ is partly smooth at $\xxs(y)$ relative to $\Mm$, then
  \begin{equation*}
    -\FxM(\xxs(y),y) = \JxM(\xxs(y)) = \e{\xxs(y)} .
  \end{equation*}
\end{lem}

\begin{proof}
The first monotone inclusion is just the first-order necessary and sufficient minimality condition for our convex program.
The second claim follows from \eqref{eq:covgrad} and Fact~\ref{fact:gradhess}. \qed
\end{proof}

\input{sections/proofs-local}

\input{sections/proofs-dof}

% Local Variables: 
% TeX-master: "../l1l2-variations-dof.tex"
% End:

%% file: sections/proofs-local.tex
\subsection{Proof of Lemma~\ref{lem:injectivity-cond}} % (fold)
The equivalence is a consequence of simple arguments from linear algebra. Indeed, when both $\Fxx(\xx,y)$ and $\Q(\xx)$ are positive semidefinite on $\T$, we have $\dotp{(\Fxx(\xx,y)\xi}{\xi} \geq 0$ and $\dotp{(\Q(\xx)\xi}{\xi} \geq 0$, $\forall ~ \xi \in \T$. Thus, for \eqref{eq-injectivity-cond} to hold, it is necessary and sufficient that $\nexists ~ 0 \neq \xi \in \T$ such that $\xi \in \Ker( \Fxx(\xx,y) )$ and $\xi \in \Ker( \Fxx(\xx,y) )$, which is exactly what we state. 

When $\Mm=\xx+\T$, the Riemannian hessians $\Fxx(\xx,y)$ and $\Q(\xx)$ are given by \eqref{eq:covhessFlin} and \eqref{eq:covhesslin}. Convexity and smoothness of $\F(\cdot,y)$ combined with \eqref{eq:covhessFlin} imply that $\Fxx(\xx,y)$ is positive semidefinite. Moreover, convexity and partial smoothness of $\J$ also yield that $\Q(\xx)$ is positive semidefinite, see~\citep[Lemma~4.6]{LiangDR15}. \qed

\subsection{Proof of Theorem~\ref{thm-local}} % (fold)
\label{sub:local}

Let $y \not\in \Hh$. To lighten the notation, we will drop the dependence of $\xxs$ on $y$, where $\xxs$ is a solution of \lasso such that $(\CondInj{\xxs}{y})$ holds. 

Let the constrained problem on $\Mm$
\eql{\label{eq:restricted}\tag{$\regulP{y}_\Mm$}
  \umin{\xx \in \Mm}
  \F( \xx, y)  +   \J(\xx) .
}
We define the notion of strong critical points that will play a pivotal role in our proof.
\begin{defn}
\label{def:strongcrit}
A point $\xxs$ is a strong local minimizer of a function $f: \Mm \to \RR \cup \ens{+\infty}$ if $f$ grows at least quadratically locally around $\xxs$ on $\Mm$, i.e. $\exists \delta > 0$ such that $f(\xx) \geq f(\xxs) + \delta\norm{\xx-\xxs}^2$, $\forall \xx \in \Mm$ near $\xxs$.
\end{defn}

The following lemma gives an equivalent characterization of strong critical points that will be more convenient in our context.
\begin{lem}
\label{lem:strongcrit}
Let $f \in \Cdeux(\Mm)$. A point $\xxs$ is a strong local minimizer of $f$ if, and only if, it is a critical  point of $f$, i.e. $\grad_\Mm f(\xxs)=0$, and satisfies the restricted positive definiteness condition
\[
\dotp{\hess_\Mm f(\xxs)\xi}{\xi} > 0 \quad \forall ~ 0 \neq \xi \in \Tgt_{\xxs}(\Mm) .
\]
\end{lem}

\begin{proof}[of Lemma~\ref{lem:strongcrit}]
The proof follows by combining the discussion after \citep[Definition~5.4]{Lewis-PartlySmooth} and \cite[Theorem~3.4]{miller2005newton}. \qed
\end{proof}

We now define the following mapping
\begin{equation*}
  \Gamma : (\xx,y) \in \Mm \times \RR^n \mapsto \FxM( \xx,y ) + \JxM(\xx).
\end{equation*}

% Note that any $\xx \in \Mc$, i.e.~$\T_{\xx}=\T$, near $\xxs$ such that $\Gamma(\xx,y) = 0$ is a solution of the constrained problem
% \eql{\label{eq:restricted}\tag{$\regulP{y}_\Mm$}
%   \umin{\alpha \in \Mm}
%   \F( \alpha, y)  +   \J(\alpha) ~.
% }
% Indeed, it can be easily checked that $\Gamma(\xx,y) = 0$ is the first-order necessary and sufficient minimality condition of \eqref{eq:restricted}.

We split the proof of the theorem in three steps.
We first show that there exists a continuously differentiable mapping $\bar y \mapsto \solmB \in \Mm$ and an open neighborhood $\neighb_y$ of $y$ such that every element $\bar y$ of $\neighb_y$ satisfies $\Gamma(\solm(\bar y), \bar y) = 0$. Then, we prove that $\solmB$ is a solution of ($\regulP{\bar y}$) for any $\bar y \in \neighb_y$. Finally, we obtain~\eqref{eq-differential} from the implicit function theorem.

%%%
\paragraph{Step~1: construction of $\solmB$.}

Using assumption \eqref{hyp-f-reg}, the sum and smooth perturbation calculus rules of partial smoothness \citep[Corollary~4.6 and Corollary~4.7]{Lewis-PartlySmooth} entail that the function $(\xx,y) \mapsto F(\xx,y) + \J(x)$ is partly smooth at $(\xxs,y)$ relative to $\Mm \times \RR^m$, which is a $\Cdeux$-manifold of $\RR^p \times \RR^m$. Moreover, it is easy to see that $\Mm \times \RR^m$ satisfies the transversality condition of \citep[Assumption~5.1]{Lewis-PartlySmooth}. By assumption $(\CondInj{\xxs}{y})$, $\xxs$ is also a strong global minimizer of \eqref{eq:restricted}, which implies in particular that $\Gamma(\xxs,y) = 0$; see Lemma~\ref{lem:strongcrit}. It then follows from \citep[Theorem~5.5]{Lewis-PartlySmooth} that there exist open neighborhoods $\widetilde \neighb_y$ of $y$ and $\widetilde \neighb_\xxs$ of $\xxs$ and a continuously differentiable mapping $\solm : \widetilde \neighb_y \to \Mm \cap \widetilde \neighb_\xxs$ such that $\solm(y) = \xxs$, and $\forall \bar y \in \widetilde \neighb_y$, $(\regulP{\bar y}_\Mm)$ has a {\emph{unique}} strong local minimizer, i.e.
\begin{align*}
    \Gamma(\solm(\bar y),\bar y)  =  0
    \qandq \text{$(\CondInj{\solm(\bar y)}{\bar y})$ holds} ,
\end{align*}
where we also used local normal sharpness property from partial smoothness of $\J$; see Fact~\ref{fact:sharp}.

%%%
\paragraph{Step~2: $\solmB$ is a solution of $(\regulP{\bar y})$.}
We now have to check the first-order optimality condition of ($\regulP{\bar y}$), i.e. that $- \Fx(\solm(\bar y),\bar y) \in \partial J(\solm(\bar y))$; see Lemma~\ref{lem:first-order}. We distinguish two cases.

\begin{enumerate}[label=$\bullet$]
\item Assume that $-\Fx(\xxs,y) \in \ri \partial J(\xxs)$. The result then follows from \citep[Theorem~5.7(ii)]{Lewis-PartlySmooth} which, moreover, allows to assert in this case that $-\Fx(\solm(\bar y),\bar y) \in \ri \partial \J(\solm(\bar y))$.

\item We now turn to the case where $-\Fx(\xxs,y) \in \rbd \partial J(\xxs)$.
Observe that $(y, \xxs) \in \Aa_{\Mm}$. In particular $y \in \Pi_{n+p,n}(\Aa_{\Mm})$.
Since by assumption $y \not\in \Hh$, one has $y \not\in \bd(\Pi_{n+p,n}(\Aa_{\Mm}))$.
Hence, there exists an open ball $\mathbb{B}(y, \epsilon)$ for some $\epsilon > 0$ such that $\mathbb{B}(y, \epsilon) \subset \Pi_{n+p,n}(\Aa_{\Mm})$.
Thus for every $\bar y \in \mathbb{B}(y, \epsilon)$, there exists $\bar \xx \in \Mc$ such that
\begin{equation*}
  - \Fx(\bar \xx, \bar y) \in \rbd \partial \J(\bar \xx) .
\end{equation*}
Since $\bar \xx \in \Mm$, $\bar \xx$ is also a critical point of $(\regulP{\bar y}_\Mm)$. But from Step~1, $\solm(\bar y)$ is unique, whence we deduce that $\solm(\bar y)=\bar \xx$. In turn, we conclude that
\begin{equation*}
  \forall \bar y \in \mathbb{B}(y, \epsilon), \quad
  - \Fx(\solm(\bar y),\bar y) \in \rbd \partial \J(\solm(\bar y)) \subset \partial \J(\solm(\bar y)) .
\end{equation*}
\end{enumerate}

%%%
\paragraph{Step~3: Computing the differential.}
In summary, we have built a mapping $\solm \in \Calt{1}(\neighb)$, with $\neighb = \widetilde \neighb_y \cap \mathbb{B}(y, \epsilon))$, such that $\solm(\bar y)$ is a solution of $(\regulP{\bar y})$ and fulfills $(\CondInj{\solm(\bar y)}{\bar y})$. We are then in position to apply the implicit function theorem to $\Gamma$, and we get the Jacobian of the mapping $\solm$ as
\begin{equation*}
    \jac \solm(\bar y) = -
    \pa{\Fxx(\solm(\bar y), \bar y) + \Q(\solm(\bar y))}^{+} \jac (\FxM)(\solm(\bar y),\bar y)
\end{equation*}
where 
\eq{
	\jac (\FxM)(\xx,y) = \proj_{\T_{\xx}} \Fxy(\xx,y),
}
where the equality is a consequence of \eqref{eq:covgrad} and linearity. \qed

%% file: sections/proofs-dof.tex
%%%%%%%%%%%%%%%%%%%%%%%%%%%%%%%%%%%%%%%%%%%%%%%%%%%%%%%%%%%%%%%%%
\subsection{Proof of Lemma~\ref{lem:unique}} % (fold)
\begin{enumerate}[label=(\roman*)]
\item See \cite[Lemma~8]{vaiter2013model}.
\item This is a specialization of Lemma~\ref{lem:injectivity-cond} using \eqref{eq-stric-cvx} and \eqref{eq:covhesslin}.  \qed
\end{enumerate}
%A first useful fact is that all solutions of \eqref{eq-group-lasso} share the same image under the action of $\XX$, which in turn implies that the prediction/response vector $\msol$ is a single-valued mapping of $y$. A proof can be found in \cite[Lemma~8]{vaiter2014model}.
%  The equivalence $(i) \iff (ii)$ comes from the following equivalent statements:
%  \begin{align*}
%    & z \in \Ker(\Fxx(\xx,y)) \cap \T \\
%    \iff& \dotp{z}{\Fxx(\xx,y)z} = \dotp{z_\T}{\hess \F(\xx,y)z_\T} = \dotp{\XX_\T z}{\Fxxo(\XX\xx,y)\XX_\T z} = 0 \\
%    \iff& z \in \Ker(\XX_\T) ~.
%  \end{align*}  
%  For the last assertion, see \cite[Lemma~8]{vaiter2013model}. \qed
% Let $\xxs_0, \xxs_1$ be two solutions of \eqref{eq-group-lasso} such that $\XX \xxs_0 \neq \XX \xxs_1$.
%   Take any convex combination $\xxs_t = (1-t) \xxs_0 + t \xxs_1$, $t \in ]0,1[$.
%   Strict convexity of $\mu \mapsto \F_0(\mu,y)$ implies that the Jensen inequality is strict, i.e.
%   \begin{equation*}
%     \F_0( \XX \xxs_t, y ) 
%     <  
%     (1-t) \F_0( \XX \xxs_0, y )
%     + 
%     t \F_0( \XX \xxs_1, y ).
%   \end{equation*}
%   The convexity of the regularization implies
%   \begin{equation*}
%     J(\xxs_t) \leq (1-t) J(\xxs_0) + t J(\xxs_1) ~.
%   \end{equation*}
%   Summing these two inequalities we arrive at 
%   \begin{equation*}
%     \F_0( \XX \xxs_t, y ) + J(\xxs_t) 
%     < 
%     \F_0( \XX \xxs_0, y ) + J(\xxs_0) 
%   \end{equation*}
%   a contradiction since $\xxs_0$ is a minimizer of \eqref{eq-group-lasso}. \qed

%%%%%%%%%%%%%%%%%%%%%%%%%%%%%%%%%%%%%%%%%%%%%%%%%%%%%%%%%%%%%%%%%
\subsection{Proof of Theorem~\ref{thm-div}} % (fold)
\label{sub:dof}
We can now prove Theorem~\ref{thm-div}. At any $y \notin \Hh \cup \Gg$, we consider $\xxs(y)$ a solution of \eqref{eq-group-lasso}. By assumption, $(\CondInj{\xxs}{y})$ holds. According to Theorem~\ref{thm-local}, one can construct a mapping $y \mapsto \solmB$ which is a solution to $(\lassoB)$, coincides with $\xxs(y)$ at $y$, and is $\Calt{1}$ for $\bar{y}$ in a neighborhood of $y$. Thus, by Lemma~\ref{lem:unique}, $\msol(\bar y)=\XX \solmB$ is a single-valued mapping, which is also $\Calt{1}$ in a neighbourhood of $y$. Moreover, its differential is equal to $\Delta(y)$ as given, where we applied the chain rule in \eqref{eq:covhessF}. \qed
%Note that this shows that this computation is independent of the particular choice of $\xxs$ provided that $(\CondInj{\xxs}{y})$ holds. 

%%%%%%%%%%%%%%%%%%%%%%%%%%%%%%%%%%%%%%%%%%%%%%%%%%%%%%%%%%%%%%%%%
\subsection{Proof of Proposition~\ref{prop-exist}} % (fold)
\label{sub:exist}
The proofs of both statements are constructive.

\begin{enumerate}[label=(\roman*)]
\item Polyhedral penalty:
any polyhedral convex $\J$ can be written as~\citep{Rockafellar96}
\begin{align*}
\J(\xx) &= \umax{i \in \ens{1,\dots,q}} \ens{\dotp{d_i}{\xx} - b_i} + \iota_{\Cc}(\xx), \\
\Cc 	&= \enscond{\xx \in \RR^p}{\dotp{a_k}{\xx} \leq c_k}, k \in \ens{1,\dots,r} .
\end{align*}
It is straightforward to show that
\begin{gather*}
\partial \J(\xx) = \co \ens{d_i}_{i \in I_{\xx}}  + \cone \ens{a_k}_{k \in K_{\xx}}, \qwhereq\\
I_{\xx} = \enscond{i}{\dotp{d_i}{\xx} - b_i = \J(\xx)} \qandq K_{\xx} = \enscond{j}{\dotp{a_j}{\xx} = c_i} ,
\end{gather*}
and
\[
\T_{\xx} = \enscond{h}{\dotp{h}{d_i} = \dotp{h}{d_j} = \tau_{\xx}, ~~ \forall i,j \in I_{\xx}} \cap \enscond{h}{\dotp{h}{a_k} = 0, ~~ \forall k \in K_{\xx}} .
\]
Let $\xxs$ be a solution of \eqref{eq-group-lasso} for $\J$ as above. Recall from Example~\ref{ex:injpolyh} that $(\CondInj{\xxs}{y})$ is equivalent to $\Ker(\XX) \cap \T_{\xxs} = \ens{0}$. Suppose that this condition does not. Thus, there exists a nonzero vector $h \in \T_{\xxs}$ such that the vector $v_t = \xxs + th$, $t \in \RR$, satisfies $\XX v_t = \XX \xxs$. Moreover,
\begin{gather*}
\dotp{v_t}{d_i} - b_i = 
\begin{cases}
\J(\xxs) + t\tau_{\xxs}, 						& \text{if } i \in I_{\xxs} \\
\dotp{\xxs}{d_i} - b_i + t\dotp{h}{d_i} < \J(\xxs) + t\dotp{h}{d_i}	& \text{otherwise} .
\end{cases}
\\\qandq \\
\dotp{v_t}{a_k} = 
\begin{cases}
c_k, 									& \text{if } k \in K_{\xxs} \\
\dotp{\xxs}{a_k} + t\dotp{h}{a_k} < c_k + t\dotp{h}{a_k}		& \text{otherwise} .
\end{cases}
\end{gather*}
Thus, for $t \in ]-t_0,t_0[$, where
\[
t_0 = \min\pa{	\min_{i \notin I_{\xxs}}\bens{\frac{\J(\xxs) - \dotp{\xxs}{d_i} + b_i}{\abs{\dotp{h}{d_i} - \tau_{\xxs}}}},
		\min_{k \notin K_{\xxs}}\bens{\frac{c_k - \dotp{\xxs}{a_k}}{\abs{\dotp{h}{a_k}}}}} ,
\]
we have $I_{v_t} = I_{\xxs}$ and $K_{v_t} = K_{\xxs}$. Moreover, $v_t \in \Cc$. Therefore, for all such $t$, we indeed have $\partial \J(v_t) = \partial \J(\xxs)$ and $\T_{v_t}=\T_{\xxs}$. Altogether, we get that
\[
-\transp{\XX}\Fxo(\XX v_t,y)=-\transp{\XX}\Fxo(\XX \xxs,y) \in \partial \J(\xxs) = \partial \J(v_t) ,
\]
i.e.~$v_t$ is a solution to \eqref{eq-group-lasso}. Thus, by Lemma~\ref{lem:unique}, we deduce that $\F_0(\XX v_t,y)=\F_0(\XX\xxs,y)$ and $\J(v_t)=\J(\xxs)$. The continuity assumption~\eqref{hyp-f-reg} yields 
\[
\F_0(\XX v_{t_0},y) = \F_0(\XX \xxs,y) .
\] 
Furthermore, since $\J$ is lsc and $v_t$ is a minimizer of \eqref{eq-group-lasso}, we have
\[
\liminf_{t \to t_0} \J(v_t) \geq \J(v_{t_0}) \geq \limsup_{t \to t_0} \J(v_t) \iff \J(v_{t_0}) = \lim_{t \to t_0} \J(v_t) = \J(\xxs) .
\]
Consequently, $v_{t_0}$ is a solution of \eqref{eq-group-lasso} such that $I_{\xxs} \subsetneq I_{v_{t_0}}$ or/and $K_{\xxs} \subsetneq K_{v_{t_0}}$, which in turn implies $\T_{v_{t_0}} \subsetneq \T_{\xxs}$. Iterating this argument, we conclude.

\item General group Lasso: 
Let $\xxs$ be a solution of \eqref{eq-group-lasso} for $\J=\norm{D^* \cdot}_{1,2}$, and $I_{\xxs}=\enscond{i}{b_i \in \Bb \tandt D^*_{b_i}\xxs \neq 0}$, i.e.~the set indexing the active blocks of $D^*\xxs$. We recall from Example~\ref{ex:gglassoM} that the partial smoothness subspace $\Mm=\T_{\xxs} = \Ker(D_{\Lambda^c}^*)$, where $\Lambda=\bs(D^*\xxs)$.

From Lemma~\ref{lem:first-order} and the subdifferential of the group Lasso, $\xxs$ is indeed a minimizer if and only if there exists $\eta \in \RR^p$ such that
\eql{
\label{eq:mincondglasso}
-\transp{\XX}\Fxo(\XX \xxs,y) + \sum_{i \in I} D_{b_i} \eta_{b_i} = 0 \qandq
\begin{cases}
\eta_{b_i} = \frac{D^*_{b_i}\xxs}{\norm{D^*_{b_i}\xxs}} & \text{if}~ i \in I_{\xxs} \\
\norm{\eta_{b_i}} \leq 1 & \text{otherwise} .
\end{cases}
}
Suppose that $(\CondInj{\xxs}{y})$ (or equivalently Lemma~\ref{lem:unique}(ii)) does not hold at $\xxs$. This is equivalent to the existence of a nonzero vector $h \in \RR^p$ in the set at the end of Example~\ref{ex:injgglasso}. Let $v_t = \xxs + t h$, for $t \in \RR$. By construction, $v_t$ obeys
\begin{gather*}
v_t \in \T_{\xxs} \iff \forall i \notin I_{\xxs}, ~ D_{b_i}^* v_t = 0 \\
\qandq \XX v_t = \XX \xxs \\
\qandq \forall i \in I_{\xxs}, \exists \mu_i \in \RR, ~ D^*_{b_i} v_t = (1+t\mu_i)D^*_{b_i} \xxs .
\end{gather*}
Let 
\[
t_0 = \min\enscond{|t|}{1+t\mu_i=0, i \in I} = \min_{i \in I_{\xxs}, \mu_i \neq 0} \abs{\mu_i}^{-1} .
\] 
For all $t \in ]-t_0,t_0[$, we have $1+t\mu_i > 0$ for $i \in I_{\xxs}$ and $I_{v_t}=I_{\xxs}$ (in fact $\T_{v_t}=\T_{\xxs}$ by~Fact~\ref{fact:sharp}), and thus
\[
\frac{D^*_{b_i} v_t}{\norm{D^*_{b_i} v_t}} = \frac{D^*_{b_i} \xxs}{\norm{D^*_{b_i}\xxs}}, \quad \forall i \in I_{v_t} .
\] 
Moreover, $-\transp{\XX}\Fxo(\XX v_t,y)=-\transp{\XX}\Fxo(\XX \xxs,y)$. Inserting the last statements in \eqref{eq:mincondglasso}, we deduce that $v_t$ is a solution of \eqref{eq-group-lasso}. 

From Lemma~\ref{lem:unique}(i), we get that $\F_0(\XX v_t,y)=\F_0(\XX\xxs,y)$ and $\norm{D^*v_t}_{1,2}=\norm{D^*\xxs}_{1,2}$. By continuity of $\F_0(\cdot,y)$ (assumption~\eqref{hyp-f-reg}), and of $\norm{\cdot}_{1,2}$ one has
\[
\F_0(\XX v_{t_0}) = \F_0(\XX \xxs) \qandq \norm{D^* v_{t_0}}_{1,2} = \norm{D^* \xxs}_{1,2} .
\]
Clearly, we have constructed a solution $v_{t_0}$ of \eqref{eq-group-lasso} such that $I_{v_{t_0}} \subsetneq I_{\xxs}$, hence $\Ker(\Q(v_{t_0})) \cap \T_{v_{t_0}} \subsetneq \Ker(\Q(\xxs)) \cap \T_{\xxs}$. Iterating this argument shows the result. \qed
\end{enumerate}

\begin{remark}
For the general group Lasso, the iterative construction is guaranteed to terminate at a non-trivial point. Indeed, if it were not the case, then eventually one would construct a solution such that $0 \neq h \in \Ker(\XX) \cap \Ker(D^*)$ leading to a contradiction with a classical condition in regularization theory. Moreover, $\Ker(\XX) \cap \Ker(D^*) = \ens{0}$ is a sufficient (and necessary in our case) condition to ensure boundedness of the set of solutions to \eqref{eq-group-lasso}.
\end{remark}

\subsection{Proof of Theorem~\ref{thm-dof}} % (fold)
\label{sub:sure}

\begin{enumerate}[label=(\roman*)]
\item We obtain this assertion by proving that all $\Hh_{\Mm}$ are of zero measure for all $\Mm$, and that the union is over a finite set, because of~\eqref{hyp-tt}.
\begin{enumerate}[label=$\bullet$]
\item Since $\J$ is definable by~\eqref{eq-condition-omin}, $\Fx(\xx,y)$ is also definable by virtue of Proposition~\ref{prop-ominimal-diffjac}. 

\item Given $\Mm \in \Mscr$ which is definable, $\Mc$ is also definable. Indeed, $\Mc$ can be equivalently written
\begin{align*}
\Mc 	&= \Mm \cap \enscond{\xx}{\exists \epsilon > 0, \forall \xx' \in \Mm \cap \mathbb{B}(\xx,\epsilon), \J \in \Cdeux(\xx')} \\
	&\cap \enscond{\xx}{\forall (u,v) \in (\partial J(\xx))^2, \dotp{u-v}{\xx'}=0, \forall \xx' \in \Tgt_{\xx}(\Mm)} \\
	&\cap \enscond{\xx}{\forall \xx_r \in \Mm \to \xx \tandt u \in \partial \J(\xx), \exists u_r \to u \text{ s.t. } u_r \in \partial\J(\xx_r)} ~.
	%&\cap \enscond{\xx}{\forall \xx_r \to \xx \text{ on } \Mm \Rightarrow \partial\J(\xx_r) \to \partial \J(\xx)} ~.
%\Mc = \enscond{\xx}{\forall \xi \in \T \text{ and } \dotp{d_i}{\alpha}=0 ~  \forall i \text{ s.t. } \dotp{d_i}{\xx}=0 \Rightarrow \xi = \alpha} ~,
\end{align*}
Each of the four sets above capture a property of partial smoothness as introduced in Definition~\ref{defn:psg}. $\Mc$ involves $\Mm$ which is definable, its tangent space (which can be shown to be definable as a mapping of $\xx$ using Proposition~\ref{prop-ominimal-diffjac}), $\partial \J$ whose graph is definable thanks to Proposition~\ref{cor-ominimal-subdiff}, continuity relations and algebraic equations, whence definability follows after interpreting the logical notations (conjunction, existence and universal quantifiers) in the first-order formula in terms of set operations, and using axioms 1-4 of definability in an o-minimal structure. 
 
\item Let $\boldsymbol{D}: \RR^p \rightrightarrows \RR^p$ the set-valued mapping whose graph is 
\[
\graph(\boldsymbol{D}) = \enscond{(\xx,\eta)}{\eta \in \ri \partial \J(\xx)} ~.
\]
From Lemma~\ref{lem-ominimal-ris}, $\graph(\boldsymbol{D})$ is definable. Since the graph $\partial \J$ is closed \citep{hiriart1996convex}, and definable (Proposition~\ref{cor-ominimal-subdiff}), the set
\[
\enscond{(\xx,\eta)}{\eta \in \rbd \partial \J(\xx)} = \graph(\partial \J) \setminus \graph(\boldsymbol{D}) ~,
\]
is also definable by axiom 1. This entails that $\Aa_{\Mm}$ is also a definable subset of $\RR^n \times \Mc$ since
\begin{align*}
\Aa_{\Mm} = (\RR^n \times \Mc \times \RR^n) &\cap \enscond{(y,\xx,\eta)}{\eta=-\Fx(\xx_T,y)} \\
					   &\cap (\RR^n \times \enscond{(\xx,\eta)}{\eta \in \rbd \partial \J(\xx)}) ~.
\end{align*}

\item By axiom 4, the canonical projection $\Pi_{n+p,n}(\Aa_{\Mm})$ is definable, and its boundary $\Hh_\T=\bd(\Pi_{n+p,n}(\Aa_{\Mm}))$ is also definable by \cite[Proposition~1.12]{coste1999omin} with a strictly smaller dimension than $\Pi_{n+p,n}(\Aa_{\Mm})$ \cite[Theorem~3.22]{coste1999omin}. 

\item We recall now from \citep[Theorem~2.10]{coste1999omin} that any definable subset $A \subset \RR^n$ in $\omin$ can be decomposed (stratified) in a disjoint finite union of $q$ subsets $C_i$, definable in $\omin$, called cells.
The dimension of $A$ is~\cite[Proposition 3.17(4)]{coste1999omin}
\begin{equation*}
  d = \umax{i \in \{1,\dots,q\}} d_i \leq n ~,
\end{equation*}
where $d_i=\dim(C_i)$. Altogether we get that
\begin{equation*}
  \dim \Hh_{\Mm}
  =
  \dim \bd(\Pi_{n+p,n}(\Aa_{\Mm}))
  <
  \dim \Pi_{n+p,n}(\Aa_{\Mm})
  =
  d
  \leq
  n
\end{equation*}
whence we deduce that $\Hh$ is of zero measure with respect to the Lebesgue measure on $\RR^n$ since the union is taken over the finite set $\Mscr$ by~\eqref{hyp-tt}.
\end{enumerate} 

\item $\F_0(\cdot,y)$ is strongly convex with modulus $\tau$ if, and only if,
\[
\F_0(\mu,y) = G(\mu,y) + \frac{\tau}{2} \norm{\mu}^2
\]
where $G(\cdot,y)$ is convex and satisfies~\eqref{hyp-f-reg}, and in particular its domain in $\mu$ is full-dimensional. Thus, \eqref{eq-group-lasso} amounts to solving
\begin{equation*}
  \umin{ \xx \in \RR^p }  \frac{\tau}{2} \norm{\XX\xx}^2 + G(\XX\xx,y) + \J(\xx) .
\end{equation*}
It can be recasted as a constrained optimization problem
\begin{equation*}
\umin{ \mu \in \RR^n, \xx \in \RR^p}  \frac{\tau}{2} \norm{\mu}^2 + G(\mu,y) + \J(\xx) ~\mathrm{s.t.}~ \mu = \XX\xx .
\end{equation*}
Introducing the image $(\XX\J)$ of $\J$ under the linear mapping $\XX$, it is equivalent to
\begin{equation}
\umin{ \mu \in \RR^n}  \frac{\tau}{2} \norm{\mu}^2 + G(\mu,y) + (\XX\J)(\mu) ~,
\end{equation}
where $(\XX\J)(\mu) = \umin{\enscond{\xx \in \RR^p}{\mu = \XX\xx}} \J(\xx)$ is the co-called pre-image of $J$ under $\XX$.
This is a proper closed convex function, which is finite on $\Span(\XX)$.
The minimization problem amounts to computing the proximal point at $0$ of $G(\cdot,y) + (\XX\J)$, which is a proper closed and convex function. Thus this point exists and is unique.

Furthermore, by assumption \eqref{cnd-hess-bord}, the difference function
\[
\F_0(\cdot,y_1)-\F_0(\cdot,y_2)=G(\cdot,y_1)-G(\cdot,y_2)
\]
is Lipschitz continuous on $\RR^p$ with Lipschitz constant $L\norm{y_1-y_2}$. It then follows from \cite[Proposition~4.32]{BonnansShapiro2000} that $\msol(\cdot)$ is Lipschitz continuous with constant $2L/\tau$. Moreover, $h$ is Lipschitz continuous, and thus so is the composed mapping $h \circ \msol(\cdot)$. From~\citep[Theorem~5, Section~4.2.3]{EvansGariepy92}, weak differentiability follows.

Rademacher theorem asserts that a Lipschitz continuous function is differentiable Lebesgue a.e. and its derivative and weak derivative coincide Lebesgue a.e., \citep[Theorem~2, Section~6.2]{EvansGariepy92}. Its weak derivative, whenever it exsist, is upper-bounded by the Lipschitz constant. Thus
\[
\EE\pa{\Big|\pd{ (h \circ \msol)_i}{y_i}(Y)\Big|} < +\infty ~.
\]

\item Now, by the chain rule~\citep[Remark, Section~4.2.2]{EvansGariepy92}, the weak derivative of $h \circ \msol(\cdot)$ at $y$ is precisely 
\[
\jac( h\circ\msol)(y))=\jac h\pa{\msol(y)}\De(y) ~.
\]
This formula is valid everywhere except on the set $\Hh \cup \Gg$ which is of Lebesgue measure zero as shown in (i). We conclude by invoking (ii) and Stein's lemma \citep{stein1981estimation} to establish unbiasedness of the estimator $\widehat \DOF$ of the DOF.

\item Plugging the DOF expression (iii) into that of the $\SURE$~\citep[Theorem~1]{stein1981estimation}, the statement follows.

\end{enumerate}
\qed

%%%%%%%%%%%%%%%%%%%%%%%%%%%%%%%%%%%%%%%%%%%%%%%%%%%%%%%%%%%%%%%%%
\subsection{Proof of Theorem~\ref{thm-dof-exp}} % (fold)
\label{sub:sureexp}

For (i)-(iii), the proof is exactly the same as in Theorem~\ref{thm-dof}.
For (iv): combining the DOF expression (iii) and \citep[Theorem~1]{eldar-gsure}, and rearranging the expression yields the stated result. \qed

% section proof -cor(end)
% Local Variables: 
% TeX-master: "../l1l2-variations-dof.tex"
% End:

%% file: sections/conclusion.tex
\section{Conclusion}

In this paper, we proposed a detailed sensitivity analysis of a class of estimators obtained by minimizing a general convex optimization problem with a regularizing penalty encoding a low complexity prior. This was achieved through the concept of partial smoothness. This allowed us to derive an analytical expression of the local variations of these estimators to perturbations of the observations, and also to prove that the set where the estimator behaves non-smoothly as a function of the observations is of zero Lebesgue measure. Both results paved the way to derive unbiased estimators of the prediction risk in two random scenarios, one of which covers the continuous exponential family. This analysis covers a large set of convex variational estimators routinely used in statistics, machine learning and imaging (most notably group sparsity and multidimensional total variation penalty). The simulation results confirm our theoretical findings and show that our risk estimator provides a viable way for automatic choice of the problem hyperparameters.

Despite its generality, there are still problems which do not fall within our settings. One can think for instance to the case of discrete (even exponential) distributions, risk estimation for non-canonical parameter of non-Gaussian distributions, non-convex regularizers, or the graphical Lasso. 

Extension to the discrete case is far from obvious, even in the independent case. One can think for instance of using identities derived by \citep{hudson1978nie,Hwang82}, but so far, provably unbiased estimates of SURE (not generalized one) are only available for linear estimators. 

If the distribution under consideration is from a continuous exponential family, so that our results apply, but one is interested in estimating the risk at a function of the canonical parameter. First, this function has to be Lipschitz continuous, and one has first to prove a formula of the corresponding SURE. So far, we are only aware of such results in the Gaussian case (hence our Theorem~\ref{thm-dof} which addresses this question precisely).

Strictly speaking, the $\lun$-penalized likelihood formulation of the graphical Lasso in \citep{YuanLin07} ((3) or (6) in that reference) does not fall within our framework. This is due to the fidelity/likelihood term which does not obey our assumptions. Note that the limitation due to fidelity/likelihood can be circumvented at the price of a quadratic approximation \citep[Section~4]{YuanLin07} also used in \citep{MeinshausenBuhlmann06}. 

Extending our results to the non-convex case would be very interesting to handle penalties such as SCAD or MCP. This would however require more sophisticated material from variational analysis. Not to mention the other difficulties inherent to non-convexity, including handling critical points (that are not necessarily minimizers even local in general), and the fact that the mapping $y \mapsto \msol(y)$ is no longer single-valued. All the above settings will be left to future work.

%An important research program would be to extend this analysis to regularizers that are partly smooth relative to smooth manifolds that are not affine nor linear subspaces. This is for instance the case of the nuclear norm (also known as the trace norm), which is partly smooth relative to the manifold of fixed rank matrices.

\begin{acknowledgements}
This work has been supported by the European Research Council (ERC project SIGMA-Vision) and Institut Universitaire de France.
\end{acknowledgements}

%% file: sections/o-minimal.tex
\section{Basic Properties of o-minimal Structures}
\label{sec-omin}

In the following results, we collect some important stability properties of o-minimal structures. To be self-contained, we also provide proofs. To the best of our knowledge, these proofs, although simple, are not reported in the literature or some of them are left as exercices in the authoritative references~\cite{van-den-Dries-omin-book,coste1999omin}. Moreover, in most proofs, to show that a subset is definable, we could just write the appropriate first-order formula (see \cite[Page~12]{coste1999omin}\cite[Section Ch1.1.2]{van-den-Dries-omin-book}), and conclude using \cite[Theorem~1.13]{coste1999omin}. Here, for the sake of clarity and avoid cryptic statements for the non-specialist, we will translate the first order formula into operations on the involved subsets, in particular projections, and invoke the above stability axioms of o-minimal structures. In the following, $n$ denotes an arbitrary (finite) dimension which is not necessarily the number of observations used previously the paper. 

\begin{lem}[Addition and multiplication]\label{lem-ominimal-summult}
	Let $f : \Om \subset \RR^n \rightarrow \RR^p$ and $g : \Omega \subset \RR^n \subset \RR^p$ be definable functions. Then their pointwise addition and multplication is also definable.
\end{lem}
\begin{proof} 
	Let $h=f+g$, and 
	\[
	B=(\Om \times \RR \times \Om \times \RR \times \Om \times \RR) \cap (\Om \times \RR \times \graph(f) \times \graph(h)) \cap S
	\]
	where $S=\enscond{(x,u,y,v,z,w)}{x=y=z, u=v+w}$ is obviously an algebraic (in fact linear) subset, hence definable by axiom 2. Axiom 1 and 2 then imply that $B$ is also definable. Let $\Pi_{3n+3p,n+p}: \RR^{3n+3p} \to \RR^{n+p}$ be the projection on the first $n+p$ coordinates. We then have
	\[
	\graph(h) = \Pi_{3n+3p,n+p}(B)
	\]
	whence we deduce that $h$ is definable by applying $3n+3p$ times axiom 4. Definability of the pointwise multiplication follows the same proof taking $u=v \cdot w$ in $S$. \qed
\end{proof}

\begin{lem}[Inequalities in definable sets]\label{lem-ominimal-ineq}
	Let $f : \Om \subset \RR^n \rightarrow \RR$ be a definable function. Then $\enscond{x \in \Om}{f(x) > 0}$, is definable. The same holds when replacing $>$ with $<$.
\end{lem}
Clearly, inequalities involving definable functions are accepted when defining definable sets.
 
There are many possible proofs of this statement.
\begin{proof}[1]
Let $B=\enscond{(x,y) \in \RR \times \RR}{f(x)=y} \cap (\Omega \times (0,+\infty)$, which is definable thanks to axioms 1 and 3, and that the level sets of a definable function are also definable. Thus
\[
\enscond{x \in \Om}{f(x) > 0} = \enscond{x \in \Om}{\exists y, f(x) = y, y > 0} = \Pi_{n+1,n}(B) ~,
\]
and we conclude using again axiom 4. \qed
\end{proof}
Yet another (simpler) proof.
\begin{proof}[2]
It is sufficient to remark that $\enscond{x \in \Om}{f(x) > 0}$ is the projection of the set $\enscond{(x,t) \in \Om \times \RR}{t^2f(x)-1 = 0}$, where the latter is definable owing to Lemma~\ref{lem-ominimal-summult}. \qed
\end{proof}

\begin{lem}[Derivative]\label{prop-ominimal-derivative}
Let $f: I \to \RR$ be a definable differentiable function on an open interval $I$ of $\RR$. Then its derivative $f': I \to \RR$ is also definable.
\end{lem}
\begin{proof} 
Let $g: (x,t) \in I \times \RR \mapsto g(x,t) = f(x+t)-f(x)$. Note that $g$ is definable function on $I \times \RR$ by Lemma~\ref{lem-ominimal-summult}. We now write the graph of $f'$ as
\[
\graph(f') = \enscond{(x,y) \in I \times \RR}{\forall \varepsilon > 0, \exists \delta > 0, \forall t \in \RR, \abs{t} < \delta, \abs{g(x,t) - yt} < \varepsilon|t|} ~.
\]
Let $C=\enscond{(x,y,v,t,\varepsilon,\delta) \in I \times \RR^5}{((x,t),v) \in \graph(g)}$, which is definable since $g$ is definable and using axiom 3. Let
\[
B = \enscond{(x,y,v,t,\varepsilon,\delta)}{t^2 < \delta^2, (v-ty)^2 < \varepsilon^2t^2} \cap C ~.
\]
The first part in $B$ is semi-algebraic, hence definable thanks to axiom 2. Thus $B$ is also definable using axiom 1. We can now write
\[
\graph(f') = \RR^3 \setminus \pa{\Pi_{5,3}\pa{\RR^5 \setminus \Pi_{6,5}(B)}} \cap (I \times \RR) ~,
\]
where the projectors and completions translate the actions of the existential and universal quantifiers. Using again axioms 4 and 1, we conclude. \qed
\end{proof}

With such a result at hand, this proposition follows immediately.
\begin{prop}[Differential and Jacobian]\label{prop-ominimal-diffjac}
Let $f=(f_1,\cdots,f_p): \Om \to \RR^p$ be a differentiable function on an open subset $\Om$ of $\RR^n$. If $f$ is definable, then so its differential mapping and its Jacobian. In particular, for each $i=1,\cdots,n$ and $j=1,\cdots,p$, the partial derivative $\partial f_i/\partial x_j: \Om \to \RR$ is definable.
\end{prop}

We provide below some results concerning the subdifferential.

\begin{prop}[Subdifferential]\label{cor-ominimal-subdiff}
Suppose that $f$ is a finite-valued convex definable function. Then for any $x \in \RR^n$, the subdifferential $\partial f(x)$ is definable.
\end{prop}
\begin{proof}
For every $x \in \RR^n$, the subdifferential $\partial f(x)$ reads
\begin{equation*}
  \partial f(x) = \enscond{\eta \in \RR^n}{f(x') \geq f(x) + \dotp{\eta}{x'-x} \quad \forall x' \in \RR^n} .
\end{equation*}
Let $K = \enscond{(\eta,x') \in \RR^n \times \RR^n}{f(x') < f(x) + \dotp{\eta}{x'-x}}$.
Hence, $\partial f(x) = \RR^n \setminus \Pi_{2n,n}(K)$.
Since $f$ is definable, the set $K$ is also definable using Lemma~\ref{lem-ominimal-summult} and~\ref{lem-ominimal-ineq}, whence definability of $\partial f(x)$ follows using axiom 4. \qed
\end{proof}

\begin{lem}[Graph of the relative interior]\label{lem-ominimal-ris}
  Suppose that $f$ is a finite-valued convex definable function.
  Then, the set
  \begin{equation*}
    \enscond{(x,\eta)}{\eta \in \ri \partial f(x)}
  \end{equation*}
  is definable.
\end{lem}
\begin{proof}
  Denote $C = \enscond{(x,\eta)}{\eta \in \ri \partial f(x)}$.
  Using the characterization of the relative interior of a convex set \citep[Theorem~6.4]{Rockafellar96}, we rewrite $C$ in the more convenient form 
  \begin{align*}
    C = \{(x,\eta) \, : \, &
      \forall u \in \RR^n,
      \forall z \in \RR^n, f(z) - f(x) \geq \dotp{u}{z-x} ,\\
      & \exists t > 1,
      \forall x' \in \RR^n,
      f(x') - f(x) \geq \dotp{(1-t) u + t \eta}{x'-x} \} .
  \end{align*}
  Let $D = \RR^n \times \RR^n \times \RR^n \times \RR^n \times (1,+\infty) \times \RR^n$ and $K$ defined as
  \begin{equation*}
    K = \enscond{(x,\eta,u,z,t,x') \!\in\! D}{f(z) - f(x) \geq \dotp{u}{z-x}), f(x') - f(x) \geq \dotp{(1-t) u + t \eta}{x'-x}} .
  \end{equation*}
  Thus,
  \begin{equation*}
    C = \RR^{2n} \setminus \Pi_{3n,2n} \left(
      \RR^{3n} \setminus \Pi_{4n,3n} \left(
        \Pi_{4n+1,4n} \left(
          \RR^{4n} \times (1,+\infty) \setminus \Pi_{5n+1,4n+1} (K)
        \right)
      \right)
    \right) ,
  \end{equation*}
  where the projectors and completions translate the actions of the existential and universal quantifiers. Using again axioms 4 and 1, we conclude. \qed
\end{proof}